\documentclass[10 pt]{article}
\usepackage[utf8]{inputenc}
\usepackage{amsmath}
\usepackage{enumerate}
\usepackage{amsfonts}
\usepackage{amssymb}
\usepackage{cancel}  
\usepackage{graphicx}
\usepackage{mathrsfs}
\usepackage{esint}
\usepackage{upref,amsthm,amsxtra,exscale}
\usepackage{cite}
\usepackage[colorlinks=true,urlcolor=blue,
citecolor=red,linkcolor=blue,linktocpage,pdfpagelabels,
bookmarksnumbered,bookmarksopen]{hyperref}
\usepackage{cleveref}
\usepackage[cm]{fullpage}

\usepackage{subcaption}
\usepackage{caption}

\newtheorem{theorem}{Theorem}[section]
\newtheorem{corollary}[theorem]{Corollary}
\newtheorem{remark}[theorem]{Remark}
\newtheorem{lemma}[theorem]{Lemma}
\newtheorem{proposition}[theorem]{Proposition}
\newtheorem{definition}[theorem]{Definition}

\numberwithin{equation}{section}

\newcommand{\R}{\mathbb R}
\newcommand{\SN}{\mathbb S^N}

 \def\cN{\mathcal{N}}

\def\r{\mathbb{R}}
\def\rn{\mathbb{R}^N}

\def\tilde{\widetilde}

\def\cN{\mathcal{N}}

\def\R{\mathbb{R}}
\def\S{\mathbb{S}}

\usepackage[dvipsnames]{xcolor}

\title{Fractional $Q$-curvature on the sphere and optimal partitions}
\author{Héctor A. Chang-Lara
, Juan Carlos Fernández \footnote{J.C. Fernández is supported by CONAHCYT grant CBF2023-2024-116 (Mexico) and  by UNAM-DGAPA-PAPIIT grant IN110225 (Mexico). }
,\ and Alberto Saldaña
\footnote{A. Saldaña is supported by CONAHCYT grant CBF2023-2024-116 (Mexico) and by UNAM-DGAPA-PAPIIT grant IN102925 (Mexico).}
}
\date{}

\begin{document}

\maketitle

\begin{abstract}
We study an optimal partition problem on the sphere, where the cost functional is associated with the fractional $Q$-curvature in terms of the conformal fractional Laplacian on the sphere. By leveraging symmetries, we prove the existence of a symmetric minimal partition through a variational approach. A key ingredient in our analysis is a new Hölder regularity result for symmetric functions in a fractional Sobolev space on the sphere. As a byproduct, we establish the existence of infinitely many solutions to a nonlocal weakly-coupled competitive system on the sphere that remain invariant under a group of conformal diffeomorphisms and we investigate the asymptotic behavior of least-energy solutions as the coupling parameters approach negative infinity.

\medskip

\noindent \textbf{Mathematics Subject Classification:} 
35B33, 
35R01 
(primary), 
35R11, 
58J40, 
58J70, 
58J90, 
35S15  
(secondary).

\medskip

\noindent \textbf{Keywords:} Conformal fractional Laplacian, fractional Yamabe problem, critical nonlinearity, cohomogeneity one action, symmetries.   

\end{abstract}


\section{Introduction}

\paragraph{The fractional $Q$-curvature problem.} Conformal geometry has given rise to many challenging problems in differential geometry and analysis, showing a deep interplay between the two areas. Among them, are the prescribed fractional $Q$-curvature problems, which are natural generalizations of the famous Yamabe problem, and consist in finding a conformal change of metric to a given Riemannian manifold in such a way that the fractional $Q$-curvature, a generalization of the scalar curvature, is constant. To be more precise, for any $s\in(0,N/2)$ and any closed Riemannian manifold $(M,g)$ of dimension $N\geq 3$, there exists a self-adjoint pseudodifferential operator $\mathscr{P}_g^s: C^\infty(M)\rightarrow C^\infty(M)$ which is the analogue of the fractional Laplacian $(-\Delta)^s$ on $\mathbb{R}^N$ (the pseudodifferential operator with Fourier symbol $|\cdot|^{2s}$) and satisfies the following conformal transformation law: for any conformal change of metric $\tilde{g}=\psi^{\frac{4}{N-2s}}g$, where $\psi\in  C^\infty(M)$ is positive, 
\begin{equation}\label{Equation:ConformalProperty}
\mathscr{P}^s_{\tilde{g}}(u) = \psi^{-\frac{N+2s}{N-2s}} \mathscr{P}_g^s(\psi u),
\end{equation}
see \cite{GJMS,DM18} and the references therein. We call $\mathscr{P}_g^s$ the conformal fractional Laplacian.

For each $s\in(0,N/2)$, we define the (fractional) \emph{$Q$-curvature} of the metric $g$ associated to the operator $\mathscr{P}_g^s$ as $Q_g^s:=\mathscr{P}_g^s(1)$. If $s=1$, then $\mathscr{P}_g^1:= -\Delta_g + \frac{N-2}{4(N-1)}R_g$ is the conformal Laplacian, where $R_g$ denotes the scalar curvature and $\Delta_g$ is the Laplace-Beltrami operator on $(M,g)$. If $s=2$, then $\mathscr{P}_g^2$ is the Paneitz-Branson operator. In this sense, $Q_g^s$ generalizes the notion of scalar curvature ($Q_g^1=R_g$) and the Paneitz-Branson $Q$-curvature.

In this setting, given a closed Riemannian manifold $(M,g)$, the problem of finding a conformal change of metric in such a way that the fractional $Q$-curvature is constant, is known as the \emph{fractional $Q$-curvature problem}, or \emph{the fractional Yamabe problem} on $(M,g)$. This problem is equivalent to solving
the nonlinear partial differential equation
\begin{equation}\label{fyp}
\mathscr{P}_g^s(u) = |u|^{2^\ast_{s}-2}u\quad\  \text{ on }M,
\end{equation}
where $u\in C^\infty(M)$ and $2^{\ast}_s := \frac{2N}{N-2s}$ is the fractional critical Sobolev exponent. The case $s=1$ corresponds to the famous Yamabe problem, whose understanding has motivated the study of critical elliptic problems for decades, and the cases  $s\in\mathbb{N}$ \cite{ChangBook} and $s\in(0,1)$ \cite{DM18} have also shown an intricate mathematical structure.

\paragraph{The conformal fractional Laplacian on the sphere.} 
There are several equivalent ways to define the conformal fractional Laplacian of order $2s\in(0,2)$ on a Riemannian manifold. A classical approach in conformal geometry involves formulating a scattering problem on an extended manifold, as studied by Graham and Zworski \cite{MR1965361}, Joshi and Sá Barreto \cite{MR3915490}, Witten \cite{MR1633012}, and Chang and González \cite{MR2737789}. Later, Caffarelli and Silvestre \cite{MR2354493} rediscovered and popularized an extension problem to compute the fractional Laplacian, making it widely known in the analysis and PDE communities\footnote{The extension technique was also documented by the probability community through the work of Molchanov and Ostrovski \cite{MR247668}.}. In the survey \cite{DM18}, González highlighted the connections between these methods. More recently, Caselli, Florit-Simon, and Serra have analysed the fractional Sobolev spaces in Riemannian manifolds in \cite{CFS24}.

In this paper, we present a different\textemdash more self-contained\textemdash approach to this operator (in the case of the round sphere $(M,g)=(\mathbb{S}^N,g)$), which is completely independent of scattering theory tools. To be more precise, we define the conformal fractional Laplacian on the sphere via the fractional Laplacian in $\rn$ and the stereographic projection $\sigma:\mathbb{S}^N\smallsetminus\{-e_{N+1}\}\rightarrow\mathbb{R}^N$ given by 
\begin{align}\label{sigma}
z=(z',z_{N+1})\mapsto \sigma(z):=\frac{z'}{1+z_{N+1}},
\end{align}
where $-e_{N+1}$ is the south pole on the sphere $\mathbb{S}^N$. Observe that $\sigma^{-1}(x)= \left(\frac{2x}{1+\vert x\vert^2}, \frac{1-\vert x\vert^2}{1+\vert x\vert^2}\right)$ for $x\in \rn$. 

Then, for $s\in(0,1),$ $u\in C^\infty(\mathbb S^N),$ and $z\in \mathbb S^N$, the conformal fractional Laplacian on the sphere is given by
\begin{align}\label{Psdef}
\mathscr{P}^s_{g}u(z) := c_{N,s} \, p.v.\int_{\SN}\frac{u(z)-u(\zeta)}{|z - \zeta|^{N+2s}}\, dV_g(\zeta) + A_{N,s}\, u(z),
\end{align}
where $|\cdot|$ denotes the euclidean norm in $\r^{N+1}$, $c_{N,s}:= 4^s \pi^{-\frac{N}{2}}\frac{\Gamma(\frac{N}{2}+s)}{\Gamma(2-s)}s(1-s),$ $A_{N,s}:=\frac{\Gamma\left(\frac{N}{2}+s\right)}{\Gamma\left(\frac{N}{2}-s\right)},$ and $V_g$ is the Riemannian volume element associated with the standard metric $g$ on the sphere. Let $H_g^s(\mathbb{S}^N)$ be the closure of $C^\infty(\mathbb{S}^N)$ in $L^2_g(\mathbb{S}^N)$ with respect to the norm 
    \begin{align}\label{Hsnorm}
\|u\|_{H_g^s(\mathbb{S}^N)}:= \left(\frac{c_{N,s}}{2} \int_{\SN}\int_{\SN} \frac{|u(z)-u(\zeta)|^2}{|z-\zeta|^{N+2s}}\, dV_g(z)\, dV_g(\zeta) + A_{N,s}\int_{\SN}u^2\, dV_g\right)^\frac{1}{2},\qquad u\in C^\infty(\mathbb{S}^N).
    \end{align}
    This norm is equivalent to the usual norm in $H_g^s(\mathbb{S}^N)$, see Corollary \ref{cor:equiv} below.
    
We have the following relationship. 
\begin{lemma}\label{lem:1}
    For $u\in C^\infty(\mathbb S^N)$, $v := u\circ \sigma^{-1}$, $z\in \mathbb S^N$, and $x=\sigma(z) \in \mathbb R^N$, it holds that
    \begin{align}\label{eq:2}
    \mathscr{P}^s_{g}u(z) = \psi_s(x)^{-\frac{N+2s}{N-2s}}(-\Delta)^s(\psi_sv)(x), 
    \qquad
    \psi_s(x):=\left(\frac{2}{1+|x|^{2}}\right)^{\frac{N-2s}{2}}.
    \end{align}
\end{lemma}
Note that \eqref{eq:2} is basically a particular case of the intertwining rule \eqref{Equation:ConformalProperty}. Lemma \ref{lem:1} is shown by direct computations in Appendix \ref{ap:sec}. This approach is also followed in \cite[Proposition 6.4.3]{DM18}, but their procedure still partially relies on scattering theory for the computation of the constant $A_{N,s}$.  Our proof is slightly different and we also show some asymptotic properties of the Fourier symbol of $\mathscr{P}^s_{g}$ that we need for our purposes (see Proposition \ref{Proposition:RestrictedEigenfunctions:intro}), without using any kind of harmonic extension procedures. We believe that this approach is more self-contained and that it is of independent interest, since it can be applied to more general nonlocal operators. 

As suggested by Lemma \ref{lem:1}, we now have the following interesting relationship: if $u\in H_g^s(\mathbb{S}^N)$ is a solution of \eqref{Equation:ConformalProperty}, then $v := u\circ \sigma^{-1}\in H^s(\rn)$ is a solution of
\begin{equation}\label{fyprn}
(-\Delta)^s v = |v|^{2^\ast_{s}-2}v\qquad \text{ in }\mathbb{R}^N.
\end{equation}

\paragraph{The optimal partition problem.}
One of the objectives of this work is to study an \emph{optimal partition problem} associated to the fractional $Q$-curvature problem \eqref{fyp} on the sphere. In general, these kinds of variational problems aim at dividing a manifold into subregions (partitions) in an optimal way, with the goal of minimizing a specific cost functional. They appear in various fields such as physics, geometry, and material sciences, and are connected to eigenvalue optimization, shape optimization, and geometric measure theory. See, for instance, the survey \cite{T24} and, in the fractional setting, see \cite{R18,TZ20,Z15,TVZ16}.

Showing the existence of an optimal partition is a difficult task, which often requires a mix of good compactness properties for minimizing sequences (of sets) and a robust regularity theory to ensure that the minimizer is an admissible partition. These problems are even more challenging when considering a cost functional associated to the fractional $Q$-curvature, which is in terms of a nonlocal equation with a critical exponent\textemdash well known for its compactness pathologies.  

To overcome these obstacles, one can use a suitable symmetry group that enhances the compactness properties of the problem (in fact, one can show that, without symmetries, a minimizer is never achieved in this setting as a consequence of the invariance under scalings of critical problems). In case of the round sphere $(\mathbb{S}^N,g)$, for positive integers $m$ and $n$  such that $m,n\geq 2$ and $m+n=N+1\geq 4$, we fix the subgroup of $O(N+1)$,
\begin{align}\label{Gdef}
G:=O(m)\times O(n),    
\end{align}
where $O(k)$ is the group of linear isometries of $\mathbb R^k.$ A significant feature of these symmetries is that the space of $[O(m)\times O(n)]$-orbits in $\mathbb{S}^N$ is one dimensional and diffeomorphic to a closed interval. This is used to avoid the lack of compactness via concentration at points. These symmetries were first used by Ding \cite{weiyue1986conformally} to establish the existence of infinitely many symmetric invariant sign-changing solutions to the Yamabe problem \eqref{fyprn} with $s=1$.

A subset $X\subset\mathbb{S}^N$ is \emph{$G$-invariant} if $\gamma z\in X$ for every $z\in X$ and $\gamma\in G$. The number of partitions is given by a fixed number $\ell\in \mathbb N$. Then, the \emph{set of admissible partitions} is
\begin{align*}
 \mathcal{P}_\ell := \Big\{   \{ U_1,\ldots, U_\ell \} \: : \: U_i \neq\emptyset \text{ is } \text{$G$-invariant, open in }\mathbb{S}^N, \text{ and } U_i\cap U_j =\emptyset 
  \text{ if } i\neq j \text{ for all }i,j =1,\ldots,\ell \Big\}.
\end{align*}
Next, we define the cost functional to be minimized among the elements in $\mathcal{P}_\ell$. A function $u:\mathbb{S}^N\to\mathbb{R}$ is $G$-\emph{invariant} if $u(\gamma z)=u(z)$ for every $\gamma\in G$ and $z\in\mathbb{S}^{N}$. For any nonempty $G$-invariant open subset $U$ of $\mathbb{S}^N$, consider the problem
\begin{equation}\label{Equation:DirichletProblemSphere}
\mathscr{P}_g^s u =  \vert u\vert^{2_{s}^\ast-2}u  \ \text{ in } U,\qquad 
u=0 \   \text{ in }\mathbb{S}^N\smallsetminus  U,\qquad 
u \text{ is } \text{$G$-invariant}.
\end{equation}
This problem has a nontrivial least energy solution (see Section~\ref{Section:ExistenceDirichlet}) which achieves the least energy level
\[
c_U :=\inf\left\{\frac{s}{N}\int_{U}\vert u\vert^{2_{s}^\ast} \;:\; u\neq 0, \ u \text{ solves~\eqref{Equation:DirichletProblemSphere} } \right\}.
\]
We use $c_U$ as cost function. Hence, the optimal partition problem that we are interested in can be rephrased as finding a minimizer that achieves
\begin{equation}\label{Equation:OptimalPartitionProblemSphere}
\inf_{ \{ U_1,\ldots, U_\ell \} \in \mathcal{P}_\ell } \sum_{i=1}^\ell c_{U_i}.
\end{equation}
We call this the \emph{$(G,\ell)$-optimal partition problem.} If this infimum is achieved, then we call its minimizer a solution to the $(G,\ell)$-optimal partition problem. Our main result states that~\eqref{Equation:OptimalPartitionProblemSphere} is achieved and characterizes its minimizer. In order to give a precise statement, we introduce some notation. Let $H_g^s(\mathbb S^N)^{G} := \{u\in H_g^s(\mathbb S^N) \ | \ \text{$u$ is $G$-invariant}\}$. We use $\mathbb{B}^d$ and $\mathbb{S}^{d-1}$ to denote the open unit ball and the unit sphere in $\mathbb{R}^d$, respectively, and we use $\cong$ whenever there is a $G$-invariant diffeomorphism between two sets.

\begin{theorem}\label{Theorem:MainOptimalPartitionPart1}
Let $N\geq 3$, $s\in(1/2,1),$ $\ell\in\mathbb{N}$, and let $G$ be as in \eqref{Gdef}. Then there is a solution $\{U_1,\ldots,U_\ell\}\in\mathcal{P}_\ell$ to the $(G,\ell)$-partition problem~\eqref{Equation:OptimalPartitionProblemSphere}. Moreover, $U_1,\ldots,U_\ell$ are smooth, connected, $\overline{U_1\cup\cdots\cup U_\ell}=\mathbb{S}^N$, and
\begin{itemize}
\item[$(i)$] $U_1\cong\mathbb{S}^{m-1}\times \mathbb{B}^{n}$, \ $U_i\cong\mathbb{S}^{m-1}\times\mathbb{S}^{n-1}\times(0,1)$ if  $i=2,\ldots,\ell-1$, and \ $U_\ell\cong\mathbb{B}^{m}\times \mathbb{S}^{n-1}$,
\item[$(ii)$] $\overline{U}_i\cap \overline{U}_{i+1}\cong\mathbb{S}^{m-1}\times\mathbb{S}^{n-1}$ and\quad $\overline{U}_i\cap \overline{U}_j=\emptyset$\, if\, $|j-i|\geq 2$.
\end{itemize}
\end{theorem}

Via the stereographic projection, the $(G,\ell)$-optimal partition problem ~\eqref{Equation:OptimalPartitionProblemSphere} is in a one-to-one correspondence with an optimal partition problem in $\rn$. To be more precise, note that each $\gamma\in G$ is an isometry of the unit sphere $\mathbb{S}^N$, and gives rise to a conformal diffeomorphism $\tilde\gamma:\rn\to\rn$ given by the conjugation $\tilde\gamma:=\sigma\circ\gamma^{-1}\circ\sigma^{-1}$.  A subset $\Omega$ of $\rn$ is called \emph{$G$-invariant} if $\tilde\gamma x\in\Omega$ for all $x\in\Omega$ and all $\gamma\in G$, and a function $v:\Omega\to\r$ is said to be \emph{$G$-invariant} if 
\begin{align*}
|\det D\tilde \gamma(x)|^\frac{1}{2^*_s}v(\tilde\gamma x)=v(x)\qquad \text{for all \ }\gamma\in G, \ x\in\Omega,    
\end{align*}
where $D\tilde \gamma$ denotes the Jacobian matrix of $\tilde\gamma.$

Let $\Omega$ be a nonempty open $G$-invariant subset of $\rn$ and consider the problem 
\begin{equation} \label{eq:3a:intro}
(-\Delta)^s v = |v|^{2_s^*-2}v \ \text{ in }\Omega, \qquad v=0\text{ in }\rn\backslash \Omega,\qquad v \ \text{ is $G$-invariant.}
\end{equation}
We show in Section \ref{Section:ExistenceDirichlet} that \eqref{eq:3a:intro} has a positive least energy solution which achieves the least energy level
\[
c_\Omega :=\inf\left\{\frac{s}{N}\int_{\Omega}\vert v\vert^{2_{s}^\ast} \;:\; v\neq 0, \ v \text{ solves~\eqref{eq:3a:intro}} \right\}.
\]
Hence, the optimal partition problem in the Euclidean space is
\begin{equation}\label{OPPRN}
\inf_{ \{\Omega_1,\ldots, \Omega_\ell \} \in \mathfrak{P}_\ell } \sum_{i=1}^\ell c_{\Omega_i},
\end{equation}
where 
\begin{align}\label{Prn}
\mathfrak{P}_\ell:=\{\{\Omega_1,\ldots,\Omega_\ell\}:\; \Omega_i\neq\emptyset \text{ is } G\text{-invariant and open in }\mathbb{R}^N\;\text{ for all } i=1,\ldots,\ell,\text{ and }\Omega_i\cap\Omega_j=\emptyset\text{ if }i\neq j\}.    
\end{align}

\begin{theorem} \label{thm:main_RNPart1}
Let $N\geq 3$, $s\in(1/2,1),$ $\ell\in\mathbb{N}$, $G$ as in \eqref{Gdef}, and let $\{U_1,\ldots,U_\ell\}$ be the solution of~\eqref{Equation:OptimalPartitionProblemSphere} given by Theorem~\ref{Theorem:MainOptimalPartitionPart1}.  Set
$\Omega_i:=\sigma(U_i)$ for $i=1\ldots,\ell-1$, $\Omega_\ell:=\sigma(U_\ell\backslash\{-e_{N+1}\})$, where $\sigma$ is the stereographic projection. Then $\{\Omega_1,\ldots,\Omega_\ell\}\in \mathfrak{P}_\ell$ achieves the infimum~\eqref{OPPRN}. Moreover, $\Omega_1,\ldots,\Omega_\ell$ are smooth and connected, $\overline{\Omega_1\cup\cdots\cup \Omega_\ell}=\mathbb{R}^{N}$, $\Omega_1,\ldots,\Omega_{\ell-1}$ are bounded, $\Omega_\ell$ is unbounded, and
\begin{itemize}
\item[$(i)$] $\Omega_1\cong\mathbb{S}^{m-1}\times \mathbb{B}^{n}$,\quad $\Omega_i\cong\mathbb{S}^{m-1}\times\mathbb{S}^{n-1}\times(0,1)$ if  $i=2,\ldots,\ell-1$, and\quad $\Omega_\ell\cong\mathbb{B}^m\times \mathbb{R}^{n-1}$,
\item[$(ii)$] $\overline{\Omega}_i\cap \overline{\Omega}_{i+1}\cong\mathbb{S}^{m-1}\times\mathbb{S}^{n-1}$ and\quad $\overline{\Omega}_i\cap \overline{\Omega}_j=\emptyset$\, if\, $|j-i|\geq 2$.
\end{itemize}
\end{theorem}

In Figure~\ref{fig:enter-label} we illustrate the shape of the optimal partitions given by Theorem~\ref{thm:main_RNPart1}.
\begin{figure}[!ht]
    \centering
    \includegraphics[width=0.5\linewidth]{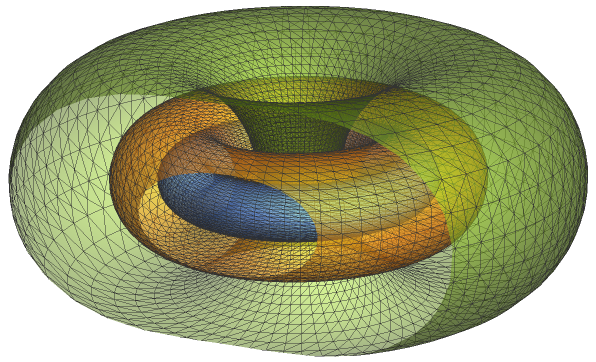}
    \caption{Illustration of the optimal partition $\{\Omega_1,\Omega_2,\Omega_3,\Omega_4\}$ of $\r^3$ given by Theorem~\ref{thm:main_RNPart1} with $\ell=4$ and $m=n=2$. Here $\Omega_1$ is the interior of the smallest torus, $\Omega_4$ is the exterior of the largest torus, and $\Omega_2$ and $\Omega_3$ are the spaces inbetween separated by the middle torus.}
    \label{fig:enter-label}
\end{figure}

\paragraph{A proof via phase separation.}
As mentioned earlier, optimal partition problems are hard to solve (even under the presence of symmetries).  To show Theorem~\ref{Theorem:MainOptimalPartitionPart1}, we use an idea coming from physics. Phase separation in Bose-Einstein condensates (BEC) occurs when two or more components of a BEC \textemdash such as different atomic species, spin states, or hyperfine states\textemdash demix, forming distinct spatial regions instead of a fully mixed state. This phenomenon arises due to the interplay between inter-component interactions and intra-component interactions. Hence, the strategy is to show first the existence of \emph{least-energy} solutions to a nonlinear competitive system (similar to the one modelling BEC but with critical exponents) and to study its limit profiles as a competition parameter goes to infinity, showing that the support of these limiting profiles gives rise to an optimal partition.  This method was introduced in the seminal paper \cite{conti2002nehari} and has been used in a wide variety of settings, including different operators, manifolds, nonlinearities, symmetries, etc., see \cite{ClappPistoiaTavares2024,ClappFernandezSaldana2021,ClappPistoia2018,CSS21,CSSV25,CFS25,FernandezPalmasTorres2024,T24} and the references therein. In particular, in \cite{ClappFernandezSaldana2021}, these ideas have been used to solve~\eqref{Equation:OptimalPartitionProblemSphere} with $s\in\mathbb{N}$. Here we adapt this strategy to the fractional case $s\in(1/2,1)$.

More specifically, consider the weakly coupled fractional competitive system
\begin{equation}\label{Equation:SystemSphere}
\mathscr{P}_g^s u_i = \vert u_i\vert^{2_{s}^\ast-2}u_i + \sum_{j\neq i}^\ell\eta_{ij} b_{ij}\vert u_j\vert^{a_{ij}} \vert u_i\vert^{b_{ij}-2}u_i \text{ on }\mathbb{S}^N,\qquad 
u_i \text{ is } \text{$G$-invariant}, \qquad  \text{ for }i=1,\ldots,\ell,
\end{equation}
where 
\begin{align*}
\eta_{ij}=\eta_{ji}<0,\qquad a_{ij}, b_{ij}>1,\qquad a_{ij}= b_{ji},\qquad a_{ij} + b_{ij} = 2_{s}^\ast,     
\end{align*}
and $\mathscr{P}_g^s $ is the conformal fractional Laplacian on $\mathbb{S}^N$. In Proposition \ref{Proposition:MainFractionalSystems}, adjusting the method in \cite{ClappSzulkin2019}, we show the existence of infinitely many solutions to \eqref{Equation:SystemSphere} and, in particular, the existence of a least-energy $G$-invariant solution to \eqref{Equation:SystemSphere} that is \emph{fully nontrivial}, namely, that each component is different from zero. The following result characterizes its limit as the competition parameter $\eta_{ij}$ goes to minus infinity, and it is of independent interest. 

\begin{theorem}\label{Theorem:MainOptimalPartition}
Let $s\in(1/2,1)$ and $\ell\in\mathbb{N}$. For each $i\neq j$, $k\in\mathbb{N}$, let $\eta_{ij,k}<0$ be such that $\eta_{ij,k}=\eta_{ji,k}$ and $\eta_{ij,k}\to -\infty$ as $k\to\infty$. Let $(u_{k,1},\ldots,u_{k,\ell})$ be a positive least energy fully nontrivial $G$-invariant solution to system~\eqref{Equation:SystemSphere} with $\eta_{ij}=\eta_{ij,k}$. Then, after passing to a subsequence, we have that $u_{k,i}\to u_{\infty,i}$ strongly in $H_g^{s}(\mathbb{S}^N)$, $u_{\infty,i}\in  C^{0}(\mathbb{S}^N)$, $u_{\infty,i}\geq 0$ and $u_{\infty,i}|_{U_i}$ is a positive least energy $G$-invariant solution of~\eqref{Equation:DirichletProblemSphere} in 
$U_i:={\{x\in\mathbb{S}^N \;:\; u_{\infty,i}(x)>0\}}$ for $i=1,\ldots,\ell.$ Moreover, up to a relabeling, $\{U_1,\ldots,U_\ell\}\in\mathcal{P}_\ell$ is the solution to the $(G,\ell)$-partition problem~\eqref{Equation:OptimalPartitionProblemSphere} given in Theorem~\ref{Theorem:MainOptimalPartitionPart1}.
\end{theorem}

The most delicate point in the proof of Theorem~\ref{Theorem:MainOptimalPartition} is to show that the sets $U_i$ are open. To argue this point, we show that all the elements in the fractional Sobolev space of $G$-invariant functions $H_g^s(\mathbb{S}^N)^G$ have a continuous representative except in an explicit set of measure zero. This implies the continuity of the limit profiles $u_{\infty,i}$ which in turn implies that the $U_i$'s are open. 

\begin{proposition}\label{Proposition:Regularity}
Let $s\in(1/2,1)$ and let $\mathcal{Z} := \left(\mathbb{S}^{m-1}\times\{0\}\right)\cup\left( \{0\}\times\mathbb{S}^{n-1}\right) \subset \mathbb{S}^N.$  For any $u\in H_g^s(\mathbb{S}^N)^G$, there exists $\widetilde{u}\in  C^0(\mathbb{S}^N\smallsetminus \mathcal{Z})$ such that $u=\widetilde{u}$ a.e. in $\mathbb{S}^N$.
\end{proposition}

This result is known for $s\in\mathbb{N}$ (see \cite[Proposition 2.4]{CSS21} and \cite[Proposition 3]{ClappFernandezSaldana2021}), where the proof strongly relies on local properties of the $H^1$-norm.  Hence, the nonlocal case requires new ideas. The proof of Proposition~\ref{Proposition:Regularity} is one of the main methodological contributions of this paper. In a nutshell, the proof relies on the following intuitive idea: the elements in $H_g^s(\mathbb{S}^N)^G$ are essentially one-dimensional, and one-dimensional functions in $H^s$ are continuous, whenever $s>\frac{1}{2}$ (we remark, however, that the space $H_g^s(\mathbb{S}^N)^G$ does contain functions which are singular at $\mathcal{Z}$, by arguing as in \cite[Remark 2.5]{CSS21}). To make this intuition rigorous, several ingredients are needed. 

\paragraph{Seizing the one-dimensional structure}
In general, exploiting the one-dimensional structure of functions in nonlocal settings is harder than it is in the local case.  Because of this, we believe it is worth explaining in detail our approach to handle $G$-invariant functions, based on symmetric spherical harmonics and properties of special polynomials. 

First, we express the elements in $H_g^1(\mathbb{S}^N)^G$ using $G$-invariant spherical harmonics, which in turn can be represented with a well-known family of orthogonal polynomials: the Jacobi polynomials.

For any $z\in \mathbb{S}^N$, we define the $G$-orbit of $z$ as the set $Gz:=\{\gamma z\;:\; \gamma\in G\}$.  Define the function 
\begin{align*}
f:\mathbb{S}^N\subset\mathbb{R}^{N+1}\equiv\mathbb{R}^m\times\mathbb{R}^{n}\rightarrow [-1,1]\qquad \text{given by}\qquad 
f(x,y) = \vert x\vert^2 - \vert y\vert^2.    
\end{align*}

The level sets of this function coincide with the $G$-orbits of points in the sphere. In particular, for every $G$-invariant function $u:\mathbb{S}^N \to \r$ there is a function $w:[-1,1]\to \r$ so that 
\begin{align}\label{iso:intro}
    u = w\circ f\qquad \text{ a.e. in }\mathbb{S}^N,
\end{align}
see Section~\ref{Section:SymmetricSobolevSPaces}. The relationship between $G$-invariant spherical harmonics and Jacobi polynomials is given in the following result. Let $\alpha:=\frac{m}{2}-1$, $\beta:=\frac{n}{2}-1,$ $P_{i}^{(\alpha,\beta)}$ be the $i$-th Jacobi polynomial (see Section \ref{sec:jacobi} for its definition), and let 
\begin{align}\label{phinp}
\Phi_{i} := P_{i}^{(\alpha,\beta)}\circ f.
\end{align}

\begin{proposition}\label{Proposition:RestrictedEigenfunctions:intro}
The sequence $(\Phi_i)$ is  an orthogonal basis of $L^2_g(\mathbb{S}^N)^G$ and of $H_g^1(\mathbb{S}^N)^G$ consisting of $G$-invariant eigenfunctions of the eigenvalue problem
\begin{equation}\label{Problem:CompleteSymmetricEigenvalueSphere}
-\Delta_g \Phi_i = \lambda_i \Phi_i \quad \text{on }\mathbb{S}^N
\end{equation}
associated to the simple eigenvalues $\lambda_{i}  = 2 i(N -1 +2 i)$ for $i\in\mathbb{N}.$ Moreover, $(\Phi_i)$ is also an orthogonal basis of $H_g^s(\mathbb{S}^N)^G$ for $s\in(0,1)$ and $\Phi_i$ is an eigenfunction for the conformal fractional Laplacian $\mathscr{P}_g^s$, namely,
\begin{equation*}
\mathscr{P}_g^s \Phi_i = \varphi_{N,s}(\lambda_i) \Phi_i \quad \text{on }\mathbb{S}^N,
\end{equation*}
where $\varphi_{N,s}$ satisfies that $C^{-1}(1+\lambda_i^{s}) \leq \varphi_{N,s}(\lambda_i)\leq C(1+\lambda_i^{s})$  for all $i\in \mathbb N_0$ and for some $C>1$.
\end{proposition}

The characterization of the eigenvalues and eigenfunctions of the Laplacian on the sphere with the symmetries imposed by $G$ was observed by Henry and Petean in \cite{HenryPetean2014} and by Catalan and Petean in \cite{catalan2022local}.  The relations between the spectra for the Laplacian on the sphere and the conformal fractional Laplacian are discussed in the survey \cite{DM18}, with an explicit formula for the function $\varphi_{N,s}$ (from this formula one can obtain the asymptotic behavior via Stirling's formula). However, to keep the presentation more self-contained, we provide a new proof, which we believe is of independent interest. Notably, our argument avoids the use of scattering theory, allowing the underlying ideas to be extended to a broader class of nonlocal operators.

Fix $\alpha:=\frac{m}{2}-1$ and $\beta:=\frac{n}{2}-1$. Let $h:[-1,1]\to \r$ be given by
\begin{align}\label{hdef}
    h(t):=C(\alpha,\beta)(1-t)^{\alpha}(1+t)^{\beta},
    \qquad C(\alpha,\beta) := \frac {|\mathbb{S}^{2\alpha+1}||\mathbb{S}^{2\beta+1}|} {2^{\alpha+\beta+2}} = \frac {|\mathbb{S}^{m-1}||\mathbb{S}^{n-1}|} {2^{\frac{m+n}{2}}},
\end{align}
where $|\mathbb{S}^{k-1}| = 2 \pi^{k/2}/\Gamma(k/2)$ is the surface measure of the sphere $\mathbb{S}^{k-1}.$  This is the weight induced by the  function $f$ via a change of variables (see the proof of Proposition \ref{Proposition:IsometricIsomorphismIntroduction}). Let $L^2_h([-1,1])$ be the $L^2$-weighted space with weight $h$ and norm 
$\Vert f\Vert_{h}^2 := \int_{-1}^1 f^2 h \, dt$ and consider the $L^2_h$-normalized Jacobi polynomials $p_i^{(\alpha,\beta)}:[-1,1]\to \r$ given by
\begin{equation}\label{Equation:NormalizedJacobiPolinomial}
p_i^{(\alpha,\beta)}:=\frac{P_i^{(\alpha,\beta)}}{\Vert P_i^{(\alpha,\beta)} \Vert_{h}}
\end{equation}
In the following, we remove the $\alpha$ and $\beta$ parameters to ease the notation of $p_i^{(\alpha,\beta)}.$  Next, we establish an isometric isomorphism between the fractional Sobolev space with symmetries $H_g^s(\mathbb{S}^N)^G$ and a suitable one-dimensional weighted space.  To be more precise, for each $s\in(0,1)$, we define the Sobolev space  $H_{h}^s([-1,1])$ as the closure of $C^\infty([-1,1])$ under the norm
\begin{align}\label{varphis}
\Vert w\Vert_{H_{h}^s} := \sqrt{\sum_{i=1}^\infty   \langle w,p_i\rangle_h^2\varphi_{N,s}(\lambda_{i})},\qquad 
 \langle w,p_i\rangle_h:=\int_{-1}^1 w\, p_i\,h\, dt,\qquad \lambda_{i}  = 2 i(N -1 +2 i).
\end{align}
The function $\varphi_{N,s}$ (given in Proposition \ref{Proposition:RestrictedEigenfunctions:intro}) is the Fourier symbol of the conformal fractional Laplacian $\mathscr{P}^s_{g}$ in the sphere, see Section \ref{sec:phi}.  Let $L^2_g(\mathbb{S}^N)^G$ denote the space of $L^2$-functions on $\mathbb{S}^N$ that are $G$-invariant and recall that $H_g^s(\mathbb{S}^N)$ is endowed with the (equivalent) norm \eqref{Hsnorm}.

\begin{proposition}\label{Proposition:IsometricIsomorphismIntroduction}
For any $s\in(0,1)$ and any $G$-invariant function $u$, let $j$ be given by $j(u)=j(w\circ f )=w$, with $w$ as in~\eqref{iso:intro}. Then $j:H_g^s(\mathbb{S}^N)^G\to H_{h}^s([-1,1])$ and $j:L^2_g(\mathbb{S}^N)^G\to L^2_{h}([-1,1])$ are isometric isomorphisms.
\end{proposition}

For $s\in(1/2,1)$, using this isometry and properties of Jacobi polynomials, we show in Section~\ref{Section:RegularityAndJacobiPolynomials} that the elements of $H_g^s(\mathbb{S}^N)^G$ satisfy a local Hölder regularity estimate (away from the singular orbit $\mathcal{Z}= \{z\in \mathbb S^N \ | \ f(z)\in\{-1,1\}\}$, but arbitrarily close), which yields Proposition~\ref{Proposition:Regularity}. 
We remark that the fractional case is particularly difficult because other approaches in the literature cannot be used in this setting.  For instance, in \cite{FernandezPetean2020}, the proofs are based on ODE techniques, which can be used once the symmetric Yamabe problem becomes an ODE. On the other hand, in \cite{CSS21}, a local argument using the $H^1$-norm is key to showing the regularity of the limit profiles. None of these tools are available for nonlocal problems. 

\medskip

To close this introduction, we comment on some open questions. As mentioned earlier, the regularity result, Proposition~\ref{Proposition:Regularity}, is crucial in our approach and it is the only place where the assumption $s>\frac{1}{2}$ is needed.  We believe that the result is also true for $s\in(0,\frac{1}{2}]$, but the regularity of the limit profiles needs to be argued differently. Perhaps the approach in \cite{TZ20} based on monotonicity formulas can be extended to our setting.  On the other hand, the assumption $s<1$ has been made for simplicity, but we believe that most of the arguments can be extended to consider the higher-order case $s\in(1,\frac{N}{2})$ (see \cite{ClappFernandezSaldana2021} for the case $s\in \mathbb N$). Furthermore, a main focus in \cite{FernandezPetean2020,CSS21} is to construct $G$-invariant solutions of the Yamabe problem that change sign exactly $\ell$-times, for any given number $\ell\in\mathbb{N}$. This is done by a ``double shooting'' method in \cite{FernandezPetean2020}, and by \textquotedblleft gluing" together the limit profiles $u_{\infty,i}$ by alternating their signs in \cite{CSS21}.  Neither of these constructions can be done in the nonlocal case (see Remark~\ref{schs}) and the existence of these solutions remains an open problem. Finally, we mention that, in recent years, continuity and differentiability in the $s$ parameter have been studied in the context of linear and nonlinear fractional problems (see \cite{HS21,AS23,JSW20,jsw24,HS22}), it would also be interesting to better understand the properties of the $s$-dependence of the optimal partitions.

\medskip

The paper is organized as follows. In Section~\ref{Section:Preliminaries}, we introduce the fractional Laplacian, the conformal fractional Laplacian on the sphere, and show some of their properties. Section~\ref{Section:SymmetricSobolevSPaces} presents the function spaces with the symmetries imposed by $G = O(m)\times O(n)$ and their connection to Jacobi polynomials. Section~\ref{regularity} is devoted to proving the crucial regularity result, Proposition~\ref{Proposition:Regularity}. In Section~\ref{Section:Variational}, we state the main variational problems studied in this paper and establish their equivalences via the stereographic projection. Section~\ref{Section:SystemsAndSegregation} is dedicated to the existence of solutions to these systems. Finally, Section~\ref{sec:euclidean} contains the proofs of the existence and qualitative properties of the optimal partition problems. 

\section{Fractional Laplacians}\label{Section:Preliminaries}

For $s\in(0,1)$, the fractional Laplacian on $\mathbb{R}^N$ is defined by
\begin{align}\label{deltasdef}
(-\Delta)^s u(x) = c_{N,s} \, p.v.\int_{\mathbb{R}^N}\frac{u(x)-u(y)}{\vert x - y\vert^{N+2s}}dy, \qquad 
c_{N,s}:=4^s \pi^{-\frac{N}{2}}\frac{\Gamma(\frac{N}{2}+s)}{\Gamma(2-s)}s(1-s).    
\end{align}
This choice of $c_{N,s}$ ensures that $(-\Delta)^s$ has the Fourier symbol $|\xi|^{2s}$. The fractional Laplacian is a positive definite self-adjoint operator that can be used to define a Hilbert space in the following way.

\begin{definition}
    Given $\Omega \subset \r^N$ open, let $D^s(\Omega)$ be the completion of $ C^\infty_c(\Omega)$ in $L^2(\mathbb{R}^N)$ with respect to the norm $\Vert u\Vert := \langle u,u \rangle^{1/2}$ defined from the inner product
    \begin{align}\label{sp:def}
    \langle u,v \rangle:=  \frac{c_{N,s}}{2}\int_{\rn}\int_{\rn}\frac{(u(x)-u(y))(v(x)-v(y))}{\vert x - y\vert^{N+2s}}\, dx\, dy, \quad u,v\in C_c^\infty(\Omega)  
    \end{align}
\end{definition}

Note that $\langle u,v \rangle= \int_{\mathbb{R}^N} v(-\Delta)^s u =\int_{\mathbb{R}^N} u(-\Delta)^s v $ for any $u,v\in C_c^\infty(\Omega)$. The next lemma characterizes the space $D^{s}(\Omega)$.

\begin{lemma}\label{A}
    Let $\Omega\subset \rn$ be an open set of class $C^0$ and let $s\in(0,1)$. Then
    \begin{align}\label{Aeq}
        D^{s}(\Omega)=\{v\in D^{s}(\R^N)\::\:v=0\text{ in }\R^N\smallsetminus\Omega\}.
    \end{align}
\end{lemma}

\begin{proof}Since $D^{s}(\Omega)\subset \{v\in D^{s}(\R^N)\::\:v=0\text{ in }\R^N\smallsetminus\Omega\}$, we just need to establish the opposite inclusion.   By \cite[Theorem 1.4.2.2]{grisvard2011elliptic}, $H^s_0(\Omega) = \{u\in H^s(\R^N)\::\: u=0 \text{ in } \R^N\smallsetminus\Omega\},$ where $H^s(\R^N)$ denotes the usual fractional Sobolev space and $H^s_0(\Omega)$ is defined as the closure of $C_c^\infty(\Omega)$ with respect to the standard $H^s$-norm.
    
    If $\Omega$ is bounded, the norm induced by the inner product~\eqref{sp:def} is equivalent to the $H^s_0(\Omega)$-norm (by the Poincaré inequality); therefore,
    \[
    D^s(\Omega) = H^s_0(\Omega) = \{u\in H^s(\R^N)\::\: u=0 \text{ in } \R^N\smallsetminus\Omega\}= \{u\in D^s(\R^N)\::\: u=0 \text{ in } \R^N\smallsetminus\Omega\}.
    \]
    
    Denote by $B_r(0)$ the ball in $\mathbb{R}^N$ with radius $r>0$ centered at the origin. If $\Omega$ is unbounded, let $\eta\in  C_c^\infty(B_2(0))$ be such that $\eta=1$ in $B_1(0)$, and set $\eta_n(x):=\eta(x/n)$. Given $u\in D^{s}(\R^N)$ with $u=0$ on $\rn\setminus \Omega$, the product $\eta_n u$ vanishes outside $\Omega_n:=\Omega\cap B_{2n}(0)$, which is a bounded domain of class $C^0$. By the previous case, we obtain that $\eta_n u\in D^{s}(\Omega_n)\subset D^{s}(\Omega)$. Clearly $\eta_n u\rightharpoonup u$ weakly in $D^{s}(\Omega)$ and, since $D^{s}(\Omega)$ is weakly closed, we conclude that $u\in D^{s}(\Omega)$. 
\end{proof}

Similar to the construction of the space $D^s(\Omega)$ from the fractional Laplacian in the Euclidean setting, we use $\mathscr{P}^s_{g}$ to define the following Hilbert space for functions defined on the sphere.

\begin{definition}
    Let $H_g^s(\mathbb{S}^N)$ be the completion of $C^\infty_c(\mathbb S^N)$ in $L^2_g(\mathbb{S}^N)$ with respect to the norm
    \begin{align}\label{Hsndef}
    \|u\|_{H_g^s(\mathbb{S}^N)} = \langle u,u \rangle_{H_g^s(\mathbb{S}^N)}^{1/2}    
    \end{align}
     defined from the inner product
    \begin{align*}
    \langle u,v \rangle_{H_g^s(\mathbb{S}^N)}:= \frac{c_{N,s}}{2} \int_{\SN}\int_{\SN} \frac{(u(z)-u(\zeta))(v(z)-v(\zeta))}{|z-\zeta|^{N+2s}}\, dV_g(z)\, dV_g(\zeta) + A_{N,s}\int_{\SN}uv\, dV_g.
    \end{align*}
\end{definition}
Note that $\langle u,v \rangle_{H_g^s(\mathbb{S}^N)}= \int_{\SN} u\mathscr{P}^s_{g}v = \int_{\SN} v\mathscr{P}^s_{g}u$ for $u,v\in C^\infty_c(\mathbb S^N).$

Now we show some isometries induced by the stereographic projection between the Lebesgue and Sobolev spaces on $\mathbb{R}^N$ and on $\mathbb S^N$.

\begin{proposition}\label{Proposition:IsometricIsomorphism}
    Given $s\in(0,1)$, the map 
    \begin{equation}\label{Equation:IsometricIsomorphism}
    \iota_s(u):=\psi_s u\circ\sigma^{-1},
    \end{equation}
    is an isometric isomorphism from $L_g^{2_{s}^\ast}(\mathbb{S}^N)$ into $L^{2_{s}^\ast}(\mathbb{R}^N)$ and from $H_g^s(\mathbb{S}^N)$ into $D^s(\mathbb{R}^N)$.
\end{proposition}

\begin{proof}
We argue as in \cite[Lemma 2.1]{CSS21}.  Notice that the volume form in the coordinates given by the stereographic projection can be written as
    \[
    dV_g = \left[  \frac{2}{1+\vert x\vert^2} \right]^N dx = \left(\left[\frac{2}{1+\vert x\vert^2}\right]^{\frac{N-2s}{2}}\right)^{\frac{2N}{N-2s}} dx = \psi_s^{2^\ast_{s}}(x) dx.
    \]
    The isometry between the Lebesgue spaces follows by density and using the following change of variables 
    \[
    \int_{\mathbb{S}^N} \vert u\vert^{2_{s}^\ast} dV_g = \int_{\mathbb{R}^N} \vert u\circ \sigma^{-1}\vert^{2_{s}^\ast} \psi_s^{2^\ast_{s}}dx= \int_{\mathbb{R}^N} \vert \iota_s(u)\vert^{2_{s}^\ast} dx\qquad \text{for $u\in  C^\infty(\mathbb{S}^N)$}.
    \]
    
    Lemma \ref{lem:1} tells us that, for any $v\in  C^\infty(\mathbb{S}^N)$, $(\mathscr{P}^s_{g} v)\circ \sigma^{-1} = \psi_s^{1-2^\ast_{s}}(-\Delta)^s(\iota_s(v)).$ Then, for $u,v\in  C^\infty(\mathbb{S}^N)$,
    \begin{equation}\label{Equation:Isometry}
    \langle u,v \rangle_{H_g^s(\mathbb{S}^N)}
    = \int_{\mathbb{S}^N\smallsetminus\{-e_{N+1}\}} u \mathscr{P}^s_{g} v
    = \int_{\mathbb{R}^N} u\circ\sigma^{-1} \psi_s^{1-2^\ast_{s}}(-\Delta)^s(\iota_s(v))\psi_s^{2_{s}^\ast}
    = \int_{\mathbb{R}^N} \iota_s(u) (-\Delta)^s\iota_s(v)
    = \langle \iota_s(u),\iota_s(v)\rangle,
    \end{equation}
    which yields the isometric isomorphism between $D^s(\mathbb{R}^N)$ and $H_g^s(\mathbb{S}^N)$ by a density argument.
\end{proof}

\begin{definition}
    Given $U\subset \mathbb{S}^N$ open, the space $H_{g,0}^s(U)$ is given as the closure of $C_c^\infty(U)$ in $H_g^s(\mathbb{S}^N)$.
\end{definition}

From this definition we obtain that $H_{g,0}^s(U)\hookrightarrow H_g^s(\mathbb{S}^N)$ is continuous. The image $\Omega:=\sigma(U\smallsetminus\{-e_{N+1}\})$ is an open set of $\rn$ which is bounded if and only if $-e_{N+1} \notin \overline U$. Once again, the map $\iota_s$ defined in~\eqref{Equation:IsometricIsomorphism} yields an isometric isomorphism 
\[
\iota_s: H_{g,0}^s(U)\rightarrow D^{s}(\Omega).
\]
 As a consequence of Lemma~\ref{A}, we get that, if $U$ is of class $C^0$, then $H_{g,0}^s(U) = \{u\in H_g^s(\mathbb{S}^N)\; : \; u = 0 \text{ a.e. in }\mathbb{S}^N\smallsetminus U\}.$ Clearly, $H_{g,0}^s(\mathbb{S}^N)=H_g^s(\mathbb{S}^N).$


\subsection{The spectrum of the conformal fractional Laplacian on the sphere}\label{sec:phi}

Next, we show that all eigenfunctions of the Laplace-Beltrami operator are eigenfunctions of $\mathscr{P}^s_{g}$ (Lemma \ref{phinslem:0}), describe the asymptotic behavior of the Fourier symbol $\varphi_{N,s}$ of $\mathscr{P}^s_{g}$ (Lemma \ref{phinslem}), and show that the norm defined in \eqref{Hsndef} is equivalent to the standard norm in $H_g^s(\mathbb S^N)$ (Corollary \ref{cor:equiv}).

These results are not new, in fact, using scattering theory and properties of hypergeometric functions, an explicit formula for $\varphi_{N,s}$ is known, namely, 
\begin{align}\label{phiex}
\varphi_{N,s}(\lambda) = \frac{\Gamma\left(\frac{1}{2}+s+\sqrt{\lambda+\left(\frac{N-1}{2}\right)^2}\right)}{\Gamma\left(\frac{1}{2}-s+\sqrt{\lambda+\left(\frac{N-1}{2}\right)^2}\right)},    
\end{align}
see, for instance, \cite[Section 6.4]{DM18}.  In this paper, however, we present new proofs   which do not rely on scattering theory or extended problems. We believe that this approach is more self-contained and of independent interest, since it can be used also for more general nonlocal operators.

First, let us recall some known results.  The standard norm in the fractional Sobolev space $H_g^s(\mathbb{S}^N)$ has different equivalent characterizations (via Fourier series, heat kernel, hypersingular integrals, and harmonic extensions), see \cite{CFS24}. Here we use the one in terms of Fourier series. Let $N\geq 2$ and recall that the eigenvalues of the spherical Laplacian $-\Delta_g$ are given by
\begin{align}\label{Equation:ExplicitEigenvalues}
    b_i:=i(i+N-1),\qquad i\in \mathbb N_0:=\mathbb N\cup \{0\}.
\end{align}
For $i\geq 2$, the corresponding eigenspaces are spanned by $c_i:= \binom{N+i}{N}-\binom{N+i-2}{N}$ orthonormal smooth real-valued spherical harmonics $Y_{i,j}\in L^2(\SN)$ for $j=1,\ldots,c_i$, namely,
\begin{align}\label{Problem:CompleteEigenvalueSphere}
    -\Delta_{g}Y_{i,j} = b_i Y_{i,j}\quad \text{ in }\SN\quad \text{ for } j=1,\ldots,c_i.
\end{align}
For the exceptional cases $i\in\{0,1\}$, we have $c_0:=1$, $c_1:=N+1$, and $Y_{i,j}$ defined as above with $Y_{0,1}$ constant and $Y_{1,j}$ the trace over the sphere of a linear function.

Denoting the Fourier coefficients of $v\in L_g^2(\SN)$ by $\widehat v(i,j):=\int_{\SN}v(\xi) Y_{i,j}(\xi)\, dV_g(\xi),$ we have that 
\begin{align*}
 \langle u,v\rangle_{\ast} := \sum_{i=0}^\infty (1+b_i^{s})\sum_{j=1}^{c_i}\widehat u(i,j) \widehat v(i,j),\qquad u,v\in H_g^s(\mathbb S^N),   
\end{align*}
is an interior product for $H_g^s(\mathbb S^N)$ inducing the standard norm in $H_g^s(\mathbb S^N)$ given by
\begin{align*}
\|u\|_{\ast} :=\left( \sum_{i=0}^\infty (1+b_i^{s})\sum_{j=1}^{c_i}|\widehat u(i,j)|^2\right)^\frac{1}{2},\qquad u\in H_g^s(\mathbb S^N).
\end{align*}

In Corollary \ref{cor:equiv} below, we show that the norms  $\|\cdot\|_{H_g^s(\mathbb{S}^N)}$ and $\|\cdot\|_{*}$ are equivalent in $H_g^s(\mathbb{S}^N)$.

Let us show that all eigenfunctions of the Laplace-Beltrami operator on the sphere are eigenfunctions of $\mathscr{P}^s_{g}$.
\begin{lemma}\label{phinslem:0}
    Given $N\geq 1$ and $s\in (0,1)$, there exists a positive function $\varphi_{N,s}$ such that, if $\phi\in C^\infty(\mathbb S^N)$ is an eigenfunction of the Laplace-Beltrami operator on the sphere with eigenvalue $\lambda$, then $\phi$ is also an eigenfunction of $\mathscr{P}^s_{g}$ with eigenvalue $\varphi_{N,s}(\lambda)$, namely,
\begin{align}
\mathscr{P}^s_{g}\phi(z)= \varphi_{N,s}(\lambda)\phi(z)\qquad \text{ for all $z\in \mathbb S^N$.}    
\end{align}
\end{lemma}
\begin{proof}
    Fix $\phi\in C^\infty(\mathbb S^N)$ an eigenfunction of the Laplace-Beltrami operator on the sphere with eigenvalue $\lambda$. We claim that 
    \begin{align}\label{appclaim}
    \mathscr{P}^s_{g} \phi (e_{N+1})= \varphi_{N,s}(\lambda)\phi(e_{N+1})    
    \end{align}
    for some positive function $\varphi_{N,s}$, and where $e_{N+1}\in\mathbb{S}^N$ denotes the north pole.  Let $\mathscr H_\lambda$ be the eigenspace of the Laplace-Beltrami operator $-\Delta_g$ on the sphere corresponding to the eigenvalue $\lambda$. Consider the linear functional $L:\mathscr H_\lambda\to \r$ given by $L(\psi)=\psi(e_{N+1})$. Since $\mathscr H_\lambda$ has finite dimension, we have that $L$ is continuous and, by Riesz representation theorem, there is $Z_\lambda \in \mathscr H_\lambda$ such that
    \begin{align}\label{duality}
    \psi(e_{N+1}) = L(\psi) =\int_{\mathbb S^N} Z_\lambda \psi \, dV_g\qquad \text{ for all }\psi\in  \mathscr H_\lambda.    
    \end{align}

    By \cite[Theorem 2.12, Section IV.2]{MR304972}, we know that $Z_\lambda$ is a function that only depends on the $(N+1)$-coordinate and, therefore, it is invariant by
    \[
    \mathcal G:=\{\rho\in O(N+1) : \rho(e_{N+1})=e_{N+1}\} \simeq O(N).
    \]
    
    By \cite[Corollary 2.9, Section IV.2]{MR304972}, we have that $Z_\lambda(e_{N+1})\neq 0$, hence the renormalization
    \[
    \Theta_\lambda := \frac{Z_\lambda}{Z_\lambda(e_{N+1})}
    \]
     is well defined. Let $Y$ denote the average of $(\phi - \phi(e_{N+1}) \Theta_\lambda)$ over $\mathcal G$, namely,
    \begin{align}\label{Ydef}
    Y(\zeta) &:= \int_{\mathcal G} (\phi(\rho\zeta) - \phi(e_{N+1}) \Theta_\lambda(\rho\zeta))\,  d\mu(\rho)   
    \qquad \text{for $\zeta\in  \mathbb S^N$,}
    \end{align}
    where $\mu$ is the Haar measure on $\mathcal G$ with $\mu(\mathcal G)=1$. The function $Y$ is also $\mathcal G$-invariant \cite[Theorem 5.14]{RudinBook}, meaning that $Y(\zeta) = Y(\xi)$ for any
    \[
    \xi \in U(\zeta_{N+1}):=\{\xi \in \mathbb S^N \::\: \xi_{N+1}=\zeta_{N+1} \}.
    \]
    Let $\zeta \in \mathbb S^N\setminus \{\pm e_{N+1}\}$. By integrating the identity $Y(\zeta) = Y(\xi)$ with respect to $\xi \in U(\zeta_{N+1})$,
    \[
    \mathfrak h^{N-1}(U(\zeta_{N+1}))Y(\zeta)=\int_{U(\zeta_{N+1})}Y(\xi)\, d\mathfrak h^{N-1}(\xi) = \int_{U(\zeta_{N+1})}\int_{\mathcal G}(\phi(\rho\xi) - \phi(e_{N+1}) \Theta_\lambda(\rho\xi)) \, d\mu(\rho) \, d\mathfrak h^{N-1}(\xi),
    \]
where $\mathfrak h^{N-1}$ denotes the $(N-1)-$Hausdorff measure.  Applying Fubini's theorem and a change of variables,
    \begin{align*}
    \mathfrak h^{N-1}(U(\zeta_{N+1}))Y(\zeta)&=\int_{\mathcal G}\int_{U(\zeta_{N+1})} (\phi(\rho\xi) - \phi(e_{N+1})\Theta_\lambda(\rho\xi))   \, d\mathfrak h^{N-1}(\xi) \,  d\mu(\rho)\\
    &=\int_{U(\zeta_{N+1})} (\phi(\xi) - \phi(e_{N+1})\Theta_\lambda(\xi))    \, d\mathfrak h^{N-1}(\xi),
    \end{align*}
     Hence,
     \begin{align}\label{Ydef2}
    Y(\zeta) = \fint_{U(\zeta_{N+1})} (\phi(\xi) - \phi(e_{N+1}) \Theta_\lambda(\xi))\,  d\mathfrak h^{N-1}(\xi)\qquad \text{for $\zeta\in  \mathbb S^N \setminus \{\pm e_{N+1}\}$.}
    \end{align}
    A similar formula also holds at the poles, with the caveat that $U(\pm 1)$ are points instead of $(N-1)$-dimensional spheres.

    By applying $-\Delta_g$ in \eqref{Ydef} or by using that $Y$ is an average of functions in the vector space $\mathscr H_\lambda$, we have that $Y\in \mathscr H_\lambda$. Since $Y$ only depends on the $(N+1)$-coordinate, we have, by \cite[Theorem 1.12, Section IV.2]{MR304972}, that there is $\alpha\in \r$ such that $Y=\alpha \Theta_\lambda$. However, $Y(e_{N+1}) = 0$ since the integrand in \eqref{Ydef2} continuously vanishes at $e_{N+1}$. Then, by \eqref{duality},
    \begin{align*}
        0 = Y(e_{N+1}) = \alpha \int_{\mathbb S^N} \Theta_\lambda Z_\lambda \, dV_g=\alpha Z_\lambda(e_{N+1}) \int_{\mathbb S^N} \Theta_\lambda^2.    
    \end{align*}

    This implies that $\alpha=0$, namely, $Y=\alpha \Theta_\lambda=0$ for every $t\in(-1,1)$, and therefore, by \eqref{Ydef},
    \begin{align}\label{a1}
    \int_{U(t)} \phi \, d\mathfrak h^{N-1} = \phi(e_{N+1}) \int_{U(t)}\Theta_\lambda \, d\mathfrak h^{N-1} \qquad \text{for $\zeta\in  \mathbb S^N$}.
    \end{align}

    Observe that, for $\zeta = (\varphi\sin\theta,\cos\theta) \in \mathbb S^N$ with $(\theta,\varphi) \in [0,\pi]\times \mathbb S^{N-1}$,
    \[
    |e_{N+1} - \zeta|^2= 2-2\cos\theta = 4\sin(\theta/2)^2.
    \]
    Using the changes of variables 
    \begin{align*}
    &[0,\pi]\times \mathbb S^{N-1}\ni (\theta,\varphi)  \mapsto \zeta = (\varphi\sin\theta,\cos\theta) \in \SN,\\
    &\mathbb S^{N-1}\times\{\cos\theta\} \ni (\varphi,\cos\theta)  \mapsto \xi=(\varphi \sin \theta,\cos\theta) \in (\mathbb S^{N-1})\sin(\theta)\times \{\cos\theta\}=U(\cos\theta)\quad \text{ for }\theta\in[0,\pi],
    \end{align*}
     we obtain that
    \begin{align*}
    &p.v.\int_{\SN}\frac{\phi(e_{N+1})-\phi(\zeta)}{|e_{N+1} - \zeta|^{N+2s}}\, dV_g(\zeta)\\ 
    &=p.v.\int_{0}^\pi\int_{\mathbb S^{N-1}}\frac{\phi(e_{N+1})-\phi(\varphi\sin\theta,\cos\theta)}{(2\sin(\theta/2))^{N+2s}}\, d\mathfrak h^{N-1}(\varphi)\, (\sin\theta)^{N-1}d\theta\\
    &=p.v.\int_0^\pi\frac{(\sin\theta)^{N-1}}{(2\sin(\theta/2))^{N+2s}}\left(\phi(e_{N+1})\mathfrak h^{N-1}(\mathbb S^{N-1})-
    \int_{\mathbb S^{N-1}}\phi(\varphi\sin\theta,\cos\theta)\, d\mathfrak h^{N-1}(\varphi)\right)\, d\theta\\
    &=p.v.\int_0^\pi \frac{(\sin\theta)^{N-1}}{(2\sin(\theta/2))^{N+2s}}\left(\phi(e_{N+1})\mathfrak h^{N-1}(\mathbb S^{N-1})-
    \frac{1}{(\sin\theta)^{N-1}}\int_{U(\cos\theta)}\phi(\xi)\, d\mathfrak h^{N-1}(\xi)\right)\, d\theta\\
    &\overset{\eqref{a1}}{=}\phi(e_{N+1}) p.v.\int_0^\pi\frac{(\sin\theta)^{N-1}}{(2\sin(\theta/2))^{N+2s}} \left(\mathfrak h^{N-1}(\mathbb S^{N-1})-
    \frac{1}{(\sin\theta)^{N-1}}\int_{U(\cos\theta)}\Theta_\lambda\, d\mathfrak h^{N-1}\right)\, d\theta\\
    &= \mu_{N,s}(\lambda)\phi(e_{N+1}),
    \end{align*}
    where 
    \[
    \mu_{N,s}(\lambda) := \mathfrak h^{N-1}(\mathbb S^{N-1}) p.v.\int_0^\pi\frac{1-\Theta_\lambda(\cos\theta)}{(2\sin(\theta/2))^{N+2s}} (\sin\theta)^{N-1} \, d\theta,
    \]
using a slight abuse of notation writing $\Theta_\lambda(t) = \Theta_\lambda(\zeta)$ for any $\zeta\in U(t)$.  Then, by \eqref{Psdef}, 
    \begin{align*}
    \mathscr{P}^s_{g} \phi (e_{N+1})= \varphi_{N,s}(\lambda)\phi(e_{N+1})\qquad \text{ with }\varphi_{N,s}(\lambda) := c_{N,s}\mu_{N,s}(\lambda) + A_{N,s}.    
    \end{align*}
    From \cite[Corollary 2.9, Section IV.2]{MR304972} it follows that $|\Theta_\lambda|\leq 1$ over the interval $[-1,1]$. This observation implies that $\mu_{N,s}\geq 0$ and, consequently, $\varphi_{N,s}(\lambda)>0$.

    To conclude, let $\widetilde\phi\in \mathscr H_\lambda$, $z\in \mathbb S^N$, and let $\rho\in O(N+1)$ be such that $\rho z = e_{N+1}$. Then, 
    $\phi:=\widetilde\phi\circ \rho^{-1} \in \mathscr H_\lambda$. By definition \eqref{Psdef} the operator $\mathscr{P}^s_{g}$ is invariant under rotations, namely, $\mathscr{P}^s_{g}\widetilde\phi(z)
    = \mathscr{P}^s_{g}\phi(e_{N+1})$. Then, by \eqref{appclaim}, 
    $\mathscr{P}^s_{g}\widetilde\phi(z)
    = \mathscr{P}^s_{g}\phi(e_{N+1})
    =  \varphi_{N,s}(\lambda)\phi(e_{N+1}) 
    =  \varphi_{N,s}(\lambda)\widetilde\phi(z),$ as claimed.
\end{proof}

Now we describe the asymptotic behavior of $\varphi_{N,s}$.

\begin{lemma}\label{phinslem}
    Let $\varphi_{N,s}$ be as in Lemma \ref{phinslem:0} and $b_i:=i(N -1 +i)$, then
    \begin{align}\label{lambdaest}
        0 < \liminf_{i \to\infty} b_i^{-s}\varphi_{N,s}(b_i)\leq \limsup_{i\to\infty} b_i^{-s}\varphi_{N,s}(b_i) < \infty.
    \end{align}
\end{lemma}
\begin{proof}
    Using the notation in the proof of Lemma \ref{phinslem:0}, we have that $\varphi_{N,s}(b_i) := c_{N,s}\mu_{N,s}(b_i) + A_{N,s}.$ To simplify notation, let $\lambda:=b_i$. Using the identity $
    \sin\left(\theta/2\right)^2 = (1 - \cos\theta)/2$ and the change of variables $t = \cos\theta$ with $d\theta = -(1-t^2)^{-1/2}dt$, we have that
    \begin{align*}
    \mu_{N,s}(\lambda)
    &= \mathfrak h^{N-1}(\mathbb S^{N-1})2^{-N/2-s} p.v.\int_{-1}^1\frac{1-\Theta_\lambda(t)}{(1-t)^{N/2+s}} (1-t^2)^{(N-1)/2}(1-t^2)^{-1/2}dt\\
    &= \mathfrak h^{N-1}(\mathbb S^{N-1})2^{-N/2-s}p.v.\int_{-1}^1\frac{1-\Theta_\lambda(t)}{(1-t)^{1+s}} (1+t)^{(N-1)/2}dt.
    \end{align*}
    
    Note that the integrand over the interval $[-1,0]$ is uniformly bounded. As $\lambda$ tends to infinity, this portion of the integral does not contribute to the asymptotic behavior of $\mu_{N,s}$. Now we just need to focus on the integration over the interval $[0,1]$, allowing us disregard the factor $(1+t)^{(N-1)/2}$, which is comparable to a constant.
    
    By the change of variables $\tau=\lambda(1-t)$ we extract the factor $\lambda^s$, i.e.,
    \[
     \int_0^1\frac{1-\Theta_\lambda(t)}{(1-t)^{1+s}} dt = \lambda^s  \int_{0}^{\lambda}\frac{1-\Theta_\lambda(1-\tau/\lambda)}{\tau^{1+s}} d\tau.
    \]
    
    The function $\Theta_\lambda$ is a polynomial solution of the Gegenbauer equation 
    \[
    (1-t^2)y''(t) - Nty'(t) + \lambda y(t) = 0,\qquad t\in[-1,1],
    \]
    which generalizes the Legendre equation and is a particular case of a Jacobi equation (see \cite[page 148]{MR304972}). 
    
    Given that $t=1$ is a regular singular point, the polynomial can be determined just from the condition $y(1)=\Theta_\lambda(e_{N+1})=1$ and the recurrence relations for the coefficients obtained from the ODE.
    
    By the change of variables $\tau=\lambda(1-t)$ and $z(\tau) = y(\lambda(1-\tau)),$ we obtain that
    \[
    (2-\lambda^{-1}\tau)\tau z''(\tau) + N(1-\lambda^{-1}\tau)z'(\tau) + z(\tau) = 0, \qquad z(0) = 1.
    \]
    As $\lambda\to \infty$, using standard arguments with Lyapunov functions (see, for instance, Lemma 3.2 in \cite{FernandezPetean2020}), we get that the solutions above converge pointwisely to the solution $z_0$ of the problem
    \[
    2\tau z''(\tau) + Nz'(\tau) + z(\tau) = 0,\qquad z(0) = 1.
    \]
    Moreover, this convergence is uniform for $\tau \in [0,1]$, meanwhile $z_0$ must be bounded in the whole interval $[0,\infty)$, being the limit of functions bounded between $\pm 1$ on the intervals $(0,\lambda)$. By dominated convergence,
    \begin{align}\label{a2}
    \lim_{\lambda\to\infty}\lambda^{-s} \int_0^1\frac{1-\Theta_\lambda(t)}{(1-t)^{1+s}} dt = \int_{0}^{\infty}\frac{1-z_0(\tau)}{\tau^{1+s}} d\tau.
    \end{align}
    Using the theory of power series solutions of ODEs, we know that $z_0$ is differentiable at $\tau=0$, with $z_0'(0)=-1/N$, and thus 
    the integral on the right in \eqref{a2} is finite. Recalling that $\lambda=b_i$, this implies that $\limsup_{i \to\infty} b_i^{-s}\mu_{N,s}(b_i) < \infty.$ 
    
    Moreover, we also know that the integrand is non-negative because it is the limit of non-negative functions. Finally, to see that the right-hand side is positive we just have to estimate the integral around zero using that $z_0'(0)=-1/N$. For $\delta>0$ sufficiently small, $z_0(\tau)<1-\tau/(2N)$ for any $\tau \in (0,\delta)$, then
    \[
    \int_{0}^{\infty}\frac{1-z_0(\tau)}{\tau^{1+s}} d\tau \geq \frac{1}{2N}\int_{0}^{\delta}\tau^{-s} d\tau = \frac{\delta^{1-s}}{2N(1-s)} > 0,
    \]
    which implies that $0 < \liminf_{i\to\infty} b_i^{-s}\mu_{N,s}(b_i).$
\end{proof}

\begin{corollary}\label{cor:equiv}
The norms $\|\cdot\|_{H_g^s(\mathbb{S}^N)}$ and $\|\cdot\|_{*}$ are equivalent in $H_g^s(\mathbb{S}^N)$. 
\end{corollary}
\begin{proof}
By Lemma \ref{phinslem:0}, we have that $\mathscr{P}^s_g\, Y_{i,j} =  \varphi_{N,s}(b_i) Y_{i,j}$.  Then, for $u\in  C^\infty(\SN)$, $    u(z)=\sum_{i=0}^\infty \sum_{j=1}^{c_i}\widehat u(i,j) Y_{i,j}(z)$ and then, by orthonormality,
\begin{align*}
\|u\|_{H_g^s(\mathbb{S}^N)}^2
&=\int_{\SN} u \mathscr{P}^s_g u \, dV_g
=\int_{\SN} 
\left(\sum_{m=0}^\infty \sum_{k=1}^{c_m}\widehat u(m,k) Y_{k,m}\right)
\left(\sum_{i=0}^\infty \sum_{j=1}^{c_i}\widehat u(i,j) \varphi_{N,s}(b_i) Y_{i,j}\right)
 \, dV_g
=
\sum_{i=0}^\infty \sum_{j=1}^{c_i}|\widehat u(i,j)|^2 \varphi_{N,s}(b_i).
\end{align*}
The equivalence between the norms $\|\cdot\|_{H_g^s(\mathbb{S}^N)}$ and $\|\cdot\|_{*}$ now follows by Lemma \ref{lambdaest} and a density argument. 
\end{proof}

\begin{remark}
The converse statement of Lemma \ref{phinslem:0} is also true, namely, if $\phi$ is an eigenfunction of $\mathscr{P}_g^s$, then also $\phi$ is an eigenfunction of the Laplace-Beltrami operator. To see this, one can use the closed expression of $\varphi_{N,s}$ given in \eqref{phinslem:0} as follows: first, observe that the function $\varphi_{N,s}$ is strictly increasing in $(0,\infty)$. Indeed, for $s\in(0,\infty)$ the function $\nu(t):=\frac{\Gamma(t+s)}{\Gamma(t-s)}$ is strictly increasing in $(s,\infty)$, because
\[
\nu'(t)=\frac{\Gamma (t+s) (\psi ^{(0)}(t+s)-\psi ^{(0)}(t-s))}{\Gamma
   (t-s)},
\]
where $\psi^{(m)}$ is the polygamma function of order $m\geq 0$; as $\psi^{(0)}$ is strictly increasing in $(0,\infty)$ (because $\psi^{(1)}(t)=\frac{d}{dt}\psi^{(0)}(t)>0$ in $(0,\infty)$), we have that $\nu'>0$ in $(s,\infty)$. \newline
\indent With this in mind, let $\phi$ be an eigenfunction of $\mathscr{P}_g^s$ with eigenvalue $\lambda$. Multiplying $\mathscr{P}_g^s \phi = \lambda \phi$ by an arbitrary spherical harmonic $Y_{i,j}$ and integrating over the sphere, we obtain that $\varphi_{N,s}(b_i) \langle Y_{i,j}, \phi \rangle_{L^2(\mathbb{S}^N)} = \lambda \langle Y_{i,j}, \phi \rangle_{L^2(\mathbb{S}^N)}$. Given that $\phi$ is non-trivial and  that $(Y_{i,j})$ is an orthonormal  basis in $L^2_g(\mathbb{S}^N)$, there must be some $i$ and $j$ such that $ \langle Y_{i,j}, \phi \rangle_{L^2(\mathbb{S}^N)}\neq 0$, for which $\varphi_{N,s}(b_i) = \lambda$.
Now, as $\varphi_{N,s}$ is continuous and strictly increasing in $(0,\infty)$, there is a unique $i_0\in\mathbb{N}_0$ such that $\varphi_{N,s}(b_{i_0}) = \lambda$, otherwise, we would obtain that $\varphi_{N,s}(i_1)=\varphi_{N,s}(i_2)$ for some $i_1\neq i_2$,  contradicting the monotonicity of $\varphi_{N,s}$. Hence, $\phi = \sum_{j=1}^{c_{i_0}} \widehat{u}(i_0,j)Y_{i_0,j}$, meaning it is also an eigenfunction for $-\Delta_g$ with eigenvalue $b_{i_0}$.
\end{remark}


\section{Spaces with symmetries}\label{Section:SymmetricSobolevSPaces}

Given $G=O(m)\times O(n)$ and a functional space $X$ over some  $G$-invariant domain $W$ in $\mathbb{S}^N$ or in $\mathbb{R}^N$, we denote the subspace of all the $G$-invariant functions in $X$ as
\begin{equation*}
    X^G:=\{ u\in X \; : \; u \text{ is $G$-invariant}\}.
\end{equation*}


As usual, the group $G$ acts isometrically on $L^2_g(\mathbb{S}^N)$ and on $H_g^s(\mathbb{S}^N)$ as $\gamma u := u\circ\gamma^{-1}$, $\gamma\in G$. We say that a function $I:H^g(\mathbb{S}^N)\rightarrow\mathbb{R}$ is $G$-invariant if $I(\gamma u)=I(u)$ for any $u\in H_g^s(\mathbb{S}^N)$ and any $\gamma\in G$.

\begin{lemma}\label{Lemma:G-invariantFunctional}
For any $s\in(0,1]$ and any $G$-invariant function $F\in L_g^2(\mathbb{S}^N)^G$, the functional $I_s: H^s_g(\mathbb{S}^N)\rightarrow\mathbb{R}$ given by
\[
I_s(u):=\frac{1}{2}\Vert u \Vert_{H_g^s(\mathbb{S}^N)}^2 - \int_{\mathbb{S}^N}u F \; dV_g
\]
is $G$-invariant.
\end{lemma}

\begin{proof}
Decompose the functional $I_s$ as $I_s= J_s - \Sigma$, where $J_s,\Sigma:H_g^s(\mathbb{S}^N)\rightarrow\mathbb{R}$ are given by $J_s(u):=\frac{1}{2}\Vert u \Vert_{H_g^s(\mathbb{S}^N)}^2$ and $\Sigma(u):=\int_{\mathbb{S}^N}u F\; dV_g$. As $F$ is $G$-invariant and as $\gamma:\mathbb{S}^N\rightarrow\mathbb{S}^N$ is an isometry, by the change of variables theorem, we have that
\[
\Sigma(\gamma u) = \int_{\mathbb{S}^N}u\circ \gamma^{-1} F \; dV_g =  \int_{\mathbb{S}^N}u\circ \gamma^{-1} F\circ\gamma^{-1} \; dV_g = \int_{\mathbb{S}^N}u F \; dV_g = \Sigma(u), 
\]
and $\Sigma$ is $G$-invariant.

Now we see that $J_s$ is $G$-invariant for any $s\in(0,1]$. The case $s=1$ is well known, see, for instance, \cite[Proposition 2]{ClappFernandezSaldana2021}. Now, for $s\in(0,1)$, notice that any $\gamma\in G$ is an isometry of $\mathbb{R}^{N+1}$, and, therefore, for any $z,\zeta\in\mathbb{S}^N$, $\vert \gamma z - \gamma \zeta\vert = \vert z - \zeta\vert.$ Then, for any $\gamma\in G$ and $u\in H_g^s(\mathbb{S}^N)$, we get that
\begin{align*}
J_s(\gamma u) &=\frac{1}{2} c_{N,s} \int_{\SN}\int_{\SN} \frac{(u(\gamma^{-1}(z))-u(\gamma^{-1}(\zeta)))^2}{|z-\zeta|^{N+2s}}\, dV_g(z)\, dV_g(\zeta) + A_{N,s}\int_{\SN}u\circ\gamma^{-1}\, dV_g \\
&=\frac{1}{2} c_{N,s} \int_{\SN}\int_{\SN} \frac{(u(\gamma^{-1}(z))-u(\gamma^{-1}(\zeta)))^2}{|\gamma^{-1}(z)-\gamma^{-1}(\zeta)|^{N+2s}}\, dV_g(z)\, dV_g(\zeta) + A_{N,s}\int_{\SN}u\circ\gamma^{-1}\, dV_g\\
&=\frac{1}{2} c_{N,s} \int_{\SN}\int_{\SN} \frac{(u(z)-u(\zeta))^2}{|z-\zeta|^{N+2s}}\, dV_g(z)\, dV_g(\zeta) + A_{N,s}\int_{\SN}u\, dV_g\\
&=J_s(u),
\end{align*}
as we wanted to prove.
\end{proof}

We use $-\Delta_g$ to denote the Laplace-Beltrami operator on the sphere.

\begin{lemma}\label{Lemma:UniquenessSymmetricProblem}
    Let $F \in L_g^2(\mathbb{S}^N)^G$ such that $\int_{\mathbb S^N} F \, dV_g=0$. Then  
    \begin{enumerate}[(i)]
    \item the problem
    \begin{equation}\label{Problem:Linear}
    -\Delta_g u = F \quad\text{ on }\ \mathbb{S}^N
    \end{equation}
    has a unique solution belonging to $H_g^1(\mathbb{S}^N)^G$ (up to adding constants);
    \item for each $s\in(0,1)$, the problem $\mathscr{P}_g^s u = F$ on $\mathbb{S}^N$ has a unique solution belonging to $H_g^s(\mathbb{S}^N)^G$ (up to adding constants).
    \end{enumerate}
\end{lemma}
\begin{proof}
    For (i), consider the functional $I_1:H_g^1(\mathbb{S}^N)\rightarrow\mathbb{R}$ given by $
    I_1(u) := \frac{1}{2}\int_{\mathbb{S}^N}\vert\nabla u\vert_g^2 dV_g - \int_{\mathbb{S}^N} uF\, dV_g.$ Observe that the solutions to the problem~\eqref{Problem:Linear} are the critical points of $I_1$. By Lemma \ref{Lemma:G-invariantFunctional}, this functional is $G$-invariant and, by the Principle of Symmetric Criticality \cite{Palais1979}, critical points of $I_1$ restricted to $H_g^1(\mathbb{S})^G$ correspond to $G$-invariant solutions to~\eqref{Problem:Linear}. By standard variational methods, this functional has, at least, one critical point in $H_g^1(\mathbb{S}^N)^G$, proving the existence of a $G$-invariant solution to the problem. By standard energy methods, every two solutions to equation \eqref{Problem:Linear} differ by a constant.
    
    The proof of (ii) is analogous, considering now, for any $s\in(0,1)$, the functional $I_s:H_g^s(\mathbb{S}^N)\rightarrow\mathbb{R}$ given by $I_s(u) := \frac{1}{2}\Vert u\Vert^2_{H_g^s(\mathbb{S}^N)} - \int_{\mathbb{S}^N} uF\, dV_g$ which is also $G$-invariant by  Lemma \ref{Lemma:G-invariantFunctional}.
\end{proof}

The Principle of Symmetric Criticality \cite{Palais1979} and the previous lemma allows us to establish the following spectral decomposition.

\begin{lemma}\label{Lemma:EigenvaluesRestrictedOperator}
    There is an orthonormal basis of $L_g^2(\mathbb{S}^N)^G$ consisting of $G$-invariant eigenfunctions of the Laplacian on the sphere. Moreover, these eigenfunctions also form a basis of $H_g^s(\mathbb{S}^N)^G$ for any $s\in(0,1]$.  
\end{lemma}
\begin{proof}
    Let
    \[
    V:=\left\{u\in H_g^1(\mathbb{S}^N)^G\; : \; \int_{\mathbb{S}^N} u \;dV_g = 0\right\}.
    \]
    By Lemma~\ref{Lemma:UniquenessSymmetricProblem} and the inverse mapping theorem, we have that the inverse of $\Delta_g\colon V\to L^2_g(\mathbb S^N)^G$ is a well defined bounded linear operator from $L^2_g(\mathbb S^N)^G$ to $V$. Let $I\colon V\to L^2_g(\mathbb S^N)^G$ be the inclusion, known to be compact. Then, $I\circ \Delta_g^{-1}\colon L^2_g(\mathbb S^N)^G\to L^2_g(\mathbb S^N)^G$ is compact and symmetric. By the Spectral Theorem \cite[Theorem D.6.7]{EvansBook}, there exists a countable orthonormal basis $(\Phi_i)$ for $L^2_g(\mathbb{S}^N)^G$ consisting of eigenfunctions of the Laplace-Beltrami operator $-\Delta_g$. 
 An eigenfunction $\Phi_i$ associated to an eigenvalue $\lambda_i$ is a critical point of a functional
    \[
    v\in H_g^1(\mathbb{S}^N)^G \mapsto \frac{1}{2}\int_{\mathbb{S}^N}\vert\nabla v\vert^2_g -\lambda_i v^2\, dV_g.
    \]
    This functional can be extended to $H_g^1(\mathbb{S}^N)$ (using the same formula) and it is a $G$-invariant functional. Then, by the Principle of Symmetric Criticality \cite{Palais1961}, $\Phi_i$ is a critical point on $H_g^1(\mathbb{S}^N)$ and
    \[
    \int_{\mathbb{S}^N} \nabla \Phi_i \cdot \nabla v \, dV_g= \lambda_i\int_{\mathbb{S}^N}  \Phi_i v\, dV_g\qquad  \text{ for every $v\in H_g^1(\mathbb{S}^N)$}.
    \]
    

    Let $u\in H_g^s(\mathbb{S}^N)$ be such that $\langle \Phi_i,u \rangle_{H_g^s(\mathbb{S}^N)} = 0$ for all $i\in \mathbb N$. Then, by Lemma \ref{phinslem:0}, $\Phi_i$ is also an eigenfunction of $\mathscr{P}_g^s$ with eigenvalue $\varphi_{N,s}(\lambda_i)$ and
    \begin{align*}
    \langle \Phi_i,u \rangle_{L^2_g(\mathbb{S}^N)}=\varphi_{N,s}(\lambda_i)^{-1}\langle \Phi_i,u \rangle_{H_g^s(\mathbb{S}^N)}=0\qquad \text{ for all $i\in \mathbb N$.}
    \end{align*}
    Since $(\Phi_i)$ is a basis in $L^2_g(\mathbb{S}^N)^G$, this implies that $u=0$, and therefore $(\Phi_i)$ is also an orthogonal basis in $H_g^s(\mathbb{S}^N).$
\end{proof}

\bigskip

We next give another description of $G$-invariant functions defined on the sphere, in terms of the polynomial $f:\mathbb S^N\subset \mathbb{R}^m\times\mathbb{R}^{n}\rightarrow\mathbb{R}$ given by
\begin{align}\label{fdef}
f(x,y) = |x|^2-|y|^2,\qquad \text{where }x\in \mathbb{R}^m \text{ and } y\in \mathbb{R}^{n}.    
\end{align}
The image of $f$ is $[-1,1]$ and each level set of $f$ determines exactly one $G$-orbit. In fact, the image under $f$ of two points in the sphere coincide if and only if they determine the same $G$-orbit, implying that $f$ is a quotient map onto $[-1,1]$. Hence, we have that $u\colon \mathbb S^N\to \R$ is $G$-invariant if and only if $u = w\circ f$ for some $w\colon [-1,1]\to \R$.

Using that $f$ is homogeneous of order 2, we get that its radial derivative $\partial_r$ satisfies that $\partial_r f = 2f$. Therefore, 
\[
\vert \nabla_g f\vert_g^2 = |\nabla f|^2-(\partial_r f)^2 = 4(1-f^2) \qquad\text{ and } \qquad \Delta_g f = \Delta f - \partial_r^2f - N\partial_r f = 2(m-n) - 2(m+n)f,
\]
see, for instance, \cite[Chapter 3]{CecilRyanBook}, \cite{FernandezPetean2020}, or \cite[proof of Lemma 2.2]{ClappFernandezSaldana2021}. If $u\in C^\infty(\mathbb{S}^N)^G$, there is $w:[0,1]\to \r$ such that $u=w\circ f$. Then, by chain rule,
\begin{align}\label{uwf}
\Delta_g u = |\nabla_g f|^2 w'' + \Delta_g f w' = 4(1-f^2)w'' + [2(m-n)-2(m+n+1)f]w'.    
\end{align}
This operator acting on $w$ naturally connects to the Jacobi polynomials, which we review next.

\subsection{Jacobi Polynomials}\label{sec:jacobi}

For any $\alpha,\beta>-1$, the $i$-th Jacobi polynomial $P_i^{(\alpha,\beta)}:\mathbb R \rightarrow\mathbb{R}$ is the unique monic polynomial satisfying the Jacobi equation
\begin{equation}\label{JacobiEquation}
(1-t^2)x'' + [\beta - \alpha - (\alpha+\beta+2)t]x' + i(i+\alpha+\beta+1)x = 0.
\end{equation}
When $\alpha$ and $\beta$ are clear from the context, the notation for $P_i^{(\alpha,\beta)}$ will be simplified to $P_i$.

With this in place, we can characterize the symmetric eigenfunctions of the Laplacian on the sphere.

\begin{proposition}\label{Proposition:RestrictedEigenfunctions}
The eigenvalues of the Laplacian on the sphere corresponding to $G$-invariant eigenfunctions are simple and given by
\[
\lambda_{i}  := 2 i(2i+N -1) \text{ for $i$ non-negative integer}.
\]
The associated eigenfunctions are
\begin{align}\label{Phii}
\Phi_{i} = P_{i}^{(\alpha,\beta)}\circ f,    
\end{align}
where $P_{i}^{(\alpha,\beta)}$ is the $i$-th Jacobi polynomial with $\alpha=\frac{m}{2}-1$ and $\beta=\frac{n}{2}-1$ and $f$ is given by \eqref{fdef}.
\end{proposition}

\begin{proof}
A direct computation using \eqref{uwf} and \eqref{JacobiEquation} yields that $\Phi_{i} = P_{i}\circ f$ is indeed an eigenfunction with the eigenvalue $\lambda_{i}:= 2i(2i+N-1)$. In this proof, we will show that this condition is also necessary. The uniqueness of the Jacobi polynomial implies that the eigenvalues are simple.

We use \cite[Lemma 3.4]{HenryPetean2014} to characterize the eigenvalues and symmetric eigenfunctions, which states the following: Let $l$ be a non-negative integer, $c\in \mathbb R$, and $F\colon \mathbb R^{N+1}\to\mathbb R$ satisfy
\[
|\nabla F|^2 = l^2|x|^{2l-2} \qquad \text{ and } \qquad \Delta F = \frac{1}{2}cl^2|x|^{l-2}.
\]
If $\Phi = w\circ F|_{\mathbb S^N}$ is an eigenfunction of the Laplacian on the sphere with eigenvalue $b_i$, then $b_i$ must be of the form $b_{li} = li(li+N-1)$ for some non-negative integer $i$. 

Applying this proposition to our case with $l=2$, $c=m-n$, and $F(x,y) = |x|^2-|y|^2$, we get that the eigenvalues must be of the form $b_{2i} = 2i(2i+N-1)$. Moreover, a $G$-invariant eigenfunction $\Phi = w\circ f$ with eigenvalue $b_{2i} = 2i(2i+N-1)$ must be such that $w$ is a solution of the following ordinary differential equation in the interval $(-1,1)$
\[
(1-t^2)w'' + \left( (\beta - \alpha) - (\alpha + \beta +2)t \right) w' + i(i+\alpha+\beta + 1) w = 0.
\]
Since the $2i$-homogeneous extension of $\Phi$ to $\mathbb R^{N+1}$ must be a harmonic polynomial, it follows that $w$ must be a multiple of the Jacobi polynomial defined by this equation.
\end{proof}

\subsection{One-dimensional weighted spaces}

Recall that $\Vert f\Vert_{h}^2 := \int_{-1}^1 f^2 h \, dt$, with $h$ as in \eqref{hdef}. The Jacobi polynomials form an $L_{h}^2([-1,1])$-orthogonal set, see \cite[Section IV.4.3]{SzegoBook}, which is normalized by
\[
p_i:=\frac{P_i}{\Vert P_i \Vert_{h}}.
\]
We also introduce the renormalization
\begin{align}\label{phi:def}
\phi_i := \frac{\Phi_{i}}{ \Vert \Phi_{i} \Vert_{L^2_g(\mathbb{S}^N)} },    
\end{align}
where $\Phi_i= P_{i}\circ f$ as in \eqref{Phii}.  Since $(\phi_i)$ is an orthonormal basis in $L_g^2(\mathbb{S}^N)^G$, a function $u \in L_g^2(\mathbb{S}^N)^G$ has the Fourier expansion
\begin{equation}\label{Lemma:Equation:FourierCoefficients}
u = \sum_{i=0}^\infty \langle u , \phi_i \rangle_{L^2_g(\mathbb S^N)}\phi_i,
\end{equation}
where the series converges in the $L_g^2(\mathbb{S}^N)$-sense. Furthermore, the second part of Lemma \ref{Lemma:EigenvaluesRestrictedOperator} also implies that if $u = w\circ f \in H^s_g(\mathbb{S}^N)^G$ for some $s\in (0,1]$, then the series also converges in the $H^s_g(\mathbb{S}^N)$-sense. In this case, we have that 
\[
\langle u, \tilde{u}\rangle_{H_g^s(\mathbb{S}^N)} = \sum_{i=0}^\infty  \langle u , \phi_i \rangle_{L^2_g(\mathbb S^N)}\langle \tilde u , \phi_i \rangle_{L^2_g(\mathbb S^N)}\varphi_{N,s}(\lambda_{i})\qquad \text{for $u = w\circ f,$ $\tilde u = \tilde w\circ f \in H^s_g(\mathbb{S}^N)^G$,}
\]
where $\varphi_{N,s}$ is the symbol of $\mathscr{P}^s_{g}$ from Lemma \ref{phinslem} and $\lambda_{i}  = 2 i(N -1 +2 i)$. The series converges absolutely, and can be bounded using the Cauchy-Schwarz inequality, namely,
\begin{align*}
\sum_{i=0}^\infty  |\langle u , \phi_i \rangle_{L^2_g(\mathbb S^N)}\langle \tilde u , \phi_i \rangle_{L^2_g(\mathbb S^N)}|\varphi_{N,s}(\lambda_{i})
&\leq
\left(\sum_{i=0}^\infty  |\langle u , \phi_i \rangle_{L^2_g(\mathbb S^N)}|^2\varphi_{N,s}(\lambda_{i})\right)^{1/2}
\left(\sum_{i=0}^\infty  |\langle \tilde u , \phi_i \rangle_{L^2_g(\mathbb S^N)}|^2\varphi_{N,s}(\lambda_{i})\right)^{1/2}\\
&= \|u\|_{H_g^s(\mathbb{S}^N)}\|\tilde u\|^2_{H_g^s(\mathbb{S}^N)}<\infty.
\end{align*}

\begin{definition}
    Given $s\in(0,1)$, let $H_{h}^s( [-1,1])$ be the closure of $C^\infty( [-1,1])$ under the norm $\Vert w\Vert_{H_{h}^s} := \langle w,w\rangle_{H_{h}^s}^{1/2}$ defined from the inner product
    \[
    \langle w,\tilde w\rangle_{H_{h}^s}:=\sum_{i=0}^\infty  \langle w,p_i\rangle_h\langle \tilde w,p_i\rangle_h\varphi_{N,s}(\lambda_{i}),\qquad 
    \langle w,p_i\rangle_h=\int_{-1}^1 w(t)p_i(t)h(t)\, dt
    \]
    with $h$ as in \eqref{hdef}.
\end{definition}

\begin{proof}[Proof of Proposition \ref{Proposition:IsometricIsomorphismIntroduction}]
Let $u\in C^\infty(\mathbb{S}^N)^G$. Then there is $w:[-1,1]\to \r$ such that $u= w\circ f.$ Using the change of variables
    \[
    (t,\theta_x,\theta_y) \in [0,\pi] \times \mathbb S^{m-1}\times \mathbb S^{n-1} \mapsto (x,y) = (\cos(t/2)\theta_x,\sin(t/2)\theta_y) \in \mathbb S^N,
    \]
    we obtain that
    \begin{align*}
        \int_{\mathbb{S}^N} |u|^2 \,dV_g 
        &=\int_{\mathbb{S}^N} |w\circ f|^2 \,dV_g
        = C(\alpha,\beta)\int_0^\pi |w(\cos\theta)|^2\sin^{m-1}\left( 
        \theta/2 \right) \cos^{n-1}\left( 
        \theta/2 \right)\; d\theta\\
&= C(\alpha,\beta)\int_{-1}^{1}w(t)^2(1-t)^{\alpha}(1+t)^{\beta} \, dt 
=\int_{-1}^{1}w(t)^2h(t)\, dt 
=\|w\|_h^2,
    \end{align*}
where $C(\alpha,\beta)$ and $h$ are given in \eqref{hdef}. Arguing similarly, 
\begin{align*}
\|u\|_{H_g^s(\mathbb{S}^N)}^2
=
\sum_{i=0}^\infty  |\langle u , \phi_i \rangle_{L^2_g(\mathbb S^N)}|^2\varphi_{N,s}(\lambda_{i})
=\sum_{i=0}^\infty  \left(\int_{-1}^{1}w(t)p_i(t)h(t)\, dt\right)^2 \varphi_{N,s}(\lambda_{i})
= \|w\|_{H_{h}^s}^2.
\end{align*}
The claim now follows by density.
\end{proof}


\section{Regularity properties of fractional Sobolev spaces with symmetries}\label{regularity}

This section is devoted to the proof of the regularity result for the symmetric Sobolev spaces given by Proposition~\ref{Proposition:Regularity}.

\subsection{Further properties of Jacobi polynomials}\label{Section:EstimatesJacobiPolynomials}

We recall some properties of the Jacobi polynomials, details can be found in Szego's book \cite{SzegoBook}.

\begin{lemma}\label{Lemma:PropertiesJacobiPolynomials}
    The Jacobi polynomials satisfy the following. 
    \begin{enumerate}
    \item Norm \cite[Formula (4.3.3)]{SzegoBook}:
    \begin{equation}\label{Eq:NormJacobiPolynomials}
    \Vert P_i^{(\alpha,\beta)}\Vert_{h}^2 = C(\alpha,\beta) \frac{2^{\alpha+\beta+1}}{2i+\alpha+\beta+1}\frac{\Gamma(i+\alpha+1)\Gamma(i+\beta+1)}{\Gamma(i+\alpha+\beta+1)\Gamma(i+1)}.
    \end{equation}
    \item Derivative \cite[Formula (4.21.7)]{SzegoBook}:
    \begin{equation}\label{Eq:DerivativeJacobiPolynomials}
    \frac{d}{dt}\left[ P_i^{(\alpha,\beta)}(t) \right] = \frac{1}{2}(i+\alpha+\beta+1)P_{i-1}^{(\alpha+1,\beta+1)}(t).
    \end{equation}
    \item Interior asymptotic behavior \cite[Theorem 8.21.8]{SzegoBook}: For every $\varepsilon \in(0,1)$,
    \begin{equation}\label{Eq:AsymptoticEstimatesJacobiPolynomials}
    \limsup_{i\to\infty}i^{1/2}\|P_i^{(\alpha,\beta)}\|_{L^\infty((1+\varepsilon,1-\varepsilon))} <\infty.
    \end{equation}
    \end{enumerate}
\end{lemma}

Next we prove two useful uniform bounds for the normalized Jacobi polynomials. First we need the following estimates involving the weighted norms of the Jacobi polynomials.

\begin{corollary}\label{lemma:BoundsJacobiPolynomials}
    For $\alpha,\beta>-1$ and $\varepsilon\in(0,1)$,
    \begin{equation}\label{Eq:BoundJacobiPolynomial1}
    \limsup_{i\to\infty} \| p_i \|_{L^\infty([-1+\varepsilon,1-\varepsilon])} <\infty,
    \end{equation}
    and
    \begin{equation}\label{Eq:BoundJacobiPolynomial2}
    \limsup_{i\to\infty} i^{-1}\| p_i' \|_{L^\infty([-1+\varepsilon,1-\varepsilon])} <\infty.
    \end{equation}
\end{corollary}

\begin{proof}
    Using the explicit expression for the norm of the Jacobi polynomials in \eqref{Eq:NormJacobiPolynomials} we get that
    \begin{align*}
     \frac{i^{-1}}{\Vert P_i^{(\alpha,\beta)}\Vert_{h}^2 } & = C(\alpha,\beta)^{-1}\frac{(2+(\alpha+\beta+1)i^{-1})}{2^{\alpha+\beta+1}}\frac{\Gamma(i+1+\alpha+\beta)\Gamma(i+1)}{\Gamma(i+1+\alpha)\Gamma(i+1+\beta)}.
    \end{align*}
    Hence, by Stirling's formula and the properties of the Gamma function,
    \begin{equation}\label{Lemma:Equation:Bound1}
    \limsup_{i\to\infty}\frac{i^{-\frac{1}{2}}}{\Vert P_i^{(\alpha,\beta)}\Vert_{h} } <\infty.
    \end{equation}
    
    The estimate \eqref{Eq:BoundJacobiPolynomial1} follows from \eqref{Eq:AsymptoticEstimatesJacobiPolynomials} and \eqref{Lemma:Equation:Bound1}, since
    \begin{align*}
    \limsup_{i\to\infty} \| p_i \|_{L^\infty([-1+\varepsilon,1-\varepsilon])}
    &= \limsup_{i\to\infty} i^{1/2}\| P_i \|_{L^\infty([-1+\varepsilon,1-\varepsilon])} \frac{i^{-1/2}}{\|P_i\|_h} <\infty.
    \end{align*}

    Finally, for the estimate \eqref{Eq:BoundJacobiPolynomial2} we now use \eqref{Eq:DerivativeJacobiPolynomials} to obtain that
    \begin{align*}
        \limsup_{i\to\infty} i^{-1}\| p_i' \|_{L^\infty([-1+\varepsilon,1-\varepsilon])}
        &= \frac{1}{2}\limsup_{i\to\infty} (1+(\alpha+\beta+1)i^{-1})\frac{\| P_{i-1}^{(\alpha+1,\beta+1)} \|_{L^\infty([-1+\varepsilon,1-\varepsilon])}}{\|P_i^{(\alpha,\beta)}\|_h}\\
        &\leq \limsup_{i\to\infty} i^{1/2}\| P_{i-1}^{(\alpha+1,\beta+1)} \|_{L^\infty([-1+\varepsilon,1-\varepsilon])}\frac{i^{-1/2}}{\|P_i^{(\alpha,\beta)}\|_h}<\infty.
    \end{align*}
\end{proof}

\subsection{Jacobi polynomials and regularity}\label{Section:RegularityAndJacobiPolynomials}

In this section we prove Proposition~\ref{Proposition:Regularity}, namely, that functions in $H_g^s(\mathbb{S}^N)^G$ are more regular. Let
\[
\mathcal{Z}:=f^{-1}(\{-1,1\}) = \left(\mathbb{S}^{m-1}\times\{0\}\right)\cup\left( \{0\}\times\mathbb{S}^{n-1}\right) \subset \mathbb{S}^N.
\]
Then,  for any $\varepsilon\in(0,1)$, $\mathcal{Z}\cap f^{-1}[-1+\varepsilon,1-\varepsilon] = \emptyset$. Recall that the Hölder seminorm in $f^{-1}[-1+\varepsilon,1-\varepsilon]$ is given by
\[
\left[ u \right]_{C^{0, s-1/2}(f^{-1}[-1+\varepsilon,1-\varepsilon])} = \sup_{z,\zeta\in f^{-1}[-1+\varepsilon,1-\varepsilon]}\frac{\vert u(z) - u(\zeta) \vert}{d_g(z,\zeta)^{s-1/2}},
\]
where $d_g$ denotes the geodesic distance.

\begin{lemma}\label{Lemma:HolderContinuity}
    For $s\in (1/2,1)$ and $\varepsilon\in (0,1)$,
    \[
    \sum_{i=0}^\infty [p_i]^2_{C^{0,s-1/2}(-1+\varepsilon,1-\varepsilon)}(1+i^2)^{-s} < \infty.
    \]
\end{lemma}

\begin{proof}
    The goal is to estimate 
    \[
    \sum_{i=0}^\infty \frac{\vert p_i(t_2) - p_i(t_1)  \vert^2}{|t_2-t_1|^{2s-1}}(1+i^2)^{-s}.
    \]
    with bounds independent of $t_1,t_2\in(-1+\varepsilon,1-\varepsilon)$ for $t_1\neq t_2$. 
    
    In the following $C>0$ denotes possibly different constants depending only on $\varepsilon$ and $s$ (and not on $i$ or $t\in(-1+\varepsilon,1-\varepsilon)$). We split the series at $|t_2-t_1|^{-1} > 1/2$. For the tail, using \eqref{Eq:BoundJacobiPolynomial1},
    \begin{align*}
    \sum_{i>|t_2-t_1|^{-1}}^\infty \frac{\vert p_i(t_2) - p_i(t_1)  \vert^2}{|t_2-t_1|^{2s-1}}(1+i^2)^{-s}
    &\leq \frac{C}{|t_2-t_1|^{2s-1}}\sum_{i>|t_2-t_1|^{-1}} (1+i^2)^{-s}
    \leq \frac{C}{|t_2-t_1|^{2s-1}} \int_{|t_2-t_1|^{-1}}^\infty \frac{dx}{(1+x^2)^s}\\
    &\leq \frac{C}{|t_2-t_1|^{2s-1}} \int_{|t_2-t_1|^{-1}}^\infty \frac{dx}{x^{2s}} = C\vert t_2-t_1\vert^{-(2s-1)} \vert t_2-t_1\vert^{2s-1} = C
    \end{align*}
    
    On the other hand, for $i\leq |t_2-t_1|^{-1}$, using~\eqref{Eq:BoundJacobiPolynomial2},
    \begin{align*}
    \sum_{0\leq i\leq |t_2-t_1|^{-1}}^\infty \frac{\vert p_i(t_2) - p_i(t_1)  \vert^2}{|t_2-t_1|^{2s-1}}(1+i^2)^{-s}
    &\leq C \vert t_2 - t_1\vert^{3-2s} \sum_{0\leq i\leq |t_2-t_1|^{-1}} (1+i^2)^{1-s}\\
    &\leq C \vert t_2 - t_1\vert^{3-2s}\vert t_2 - t_1\vert^{-1} |t_2-t_1|^{-2(1-s)}= C.
    \end{align*}
\end{proof}

Recall that $\lambda_i:= 2 i(N -1 +2 i).$

\begin{corollary}\label{cor:HolderContinuity}
    For $s\in (1/2,1)$, $\varepsilon\in (0,1)$, and $U_\varepsilon := f^{-1}((-1+\varepsilon,1-\varepsilon))$, we have that
    \[
    \sum_{i=0}^\infty \frac{[\phi_i]^2_{C^{0,s-1/2}(U_\varepsilon)}}{\varphi_{N,s}(\lambda_{i})} < \infty.
    \]
\end{corollary}

\begin{proof}
The claim follows from Lemma \ref{Lemma:HolderContinuity} using that $\phi_i = p_i\circ f$ and Lemma \ref{phinslem}. Indeed, for $z,\zeta\in U_\varepsilon$ with $z\neq \zeta$, we obtain that
    \begin{align*}
        \sum_{i=0}^\infty  \frac{\vert  \phi_i(z) - \phi_i(\zeta) \vert^2}{|z-\zeta|^{2s-1}}\varphi_{N,s}(\lambda_{i})^{-1}
        &\leq C[f]_{C^{0,1}(\mathbb S^N)}\sum_{i=0}^\infty  \frac{\vert  p_i(f(z)) - p_i(f(\zeta)) \vert^2}{|f(z)-f(\zeta)|^{2s-1}}(1+i^2)^{-s}< \infty.
    \end{align*}
\end{proof}

\begin{proof}[Proof of Proposition~\ref{Proposition:Regularity}]

Let $u\in H_g^s(\mathbb{S}^N)^G$ be given by $u = \sum_{i=0}^k \langle u, \phi_i \rangle_{L^2_g(\mathbb S^N)}\phi_i.$ We establish first a Hölder regularity estimate for $u$, independent of $k$. Given $z,\zeta\in U_{\varepsilon}$, we have that
\begin{align*}
\frac{\vert u(z) - u(\zeta) \vert^2}{|z-\zeta|^{2s-1}} 
&\leq \left(\sum_{i=0}^k \langle u, \phi_i \rangle_{L^2_g(\mathbb S^N)}^2 \varphi_{N,s}(\lambda_{i})\right) \left(\sum_{i=0}^\infty \frac{[\phi_i]^2_{C^{0,s-1/2}(U_\varepsilon)}}{\varphi_{N,s}(\lambda_{i})}\right) = \Vert u \Vert_{H_g^s(\mathbb{S}^N)}^2 \left(\sum_{i=0}^\infty \frac{[\phi_i]^2_{C^{0,s-1/2}(U_\varepsilon)}}{\varphi_{N,s}(\lambda_{i})}\right).
\end{align*}
By Corollary \ref{cor:HolderContinuity}, we conclude that, for some $C=C(\varepsilon)>0$ independent of $k$, $[u]_{C^{0,s-1/2}(U_\varepsilon)} \leq C\Vert u \Vert_{H_g^s(\mathbb{S}^N)}.$

By the density of smooth functions in $H_g^s(\mathbb{S}^N)$ we conclude that, for every $u \in H_g^s(\mathbb{S}^N)$ and $\varepsilon\in(0,1)$, there exists some $u_\varepsilon \in C^{0,s-1/2}(U_\varepsilon)$ such that $u_\varepsilon = u$ a.e. in $\mathbb S^N$. This implies the desired representation in $\mathbb S^N\setminus \mathcal Z$ and concludes the proof.
\end{proof}


\section{Variational problems with symmetries}\label{Section:Variational}

Now we set up the different variational settings that we use in our analysis.

\subsection{Fractional Yamabe equations on the sphere}

For any smooth $G$-invariant open subset $U$  of $\mathbb{S}^N$, we say that a function $u\in H_{g,0}^s(U)^G$ is a $G$-invariant \emph{weak solution} of \eqref{Equation:DirichletProblemSphere} if 
\[
\langle u,v \rangle_{H_g^s(\mathbb{S}^N)} = \int_{\mathbb{S}^N}\vert u\vert^{2^\ast_s - 2}u v \; dV_g\qquad \text{for any }v\in H_{g,0}^s(U).
\]
If $U=\mathbb{S}^N$, we recover the symmetric fractional Yamabe problem on the sphere
\begin{equation}\label{Equation:FractionalYamabeSphere}
\mathscr{P}_g^s(u) = |u|^{2^\ast_{s}-2}u\quad \text{ on }\mathbb{S}^N,\qquad u\in H_{g}^s(   \mathbb{S}^N    )^G.
\end{equation}
Accordingly, $u\in H_{g}^s(\mathbb{S}^N)^G$ is a $G$-invariant \emph{weak solution} of \eqref{Equation:FractionalYamabeSphere} if 
\[
\langle u,v \rangle_{H_g^s(\mathbb{S}^N)} = \int_{\mathbb{S}^N}\vert u\vert^{2^\ast_s - 2}u v \; dV_g\qquad \text{for any }v\in H_{g}^s(\mathbb{S}^N).
\]

Define the functional $J_U:H_{g,0}^s(U)\rightarrow\mathbb{R}$ given by
\begin{align}\label{J:def}
J_U(u):= \frac{1}{2}\Vert u\Vert^2_{H_g^s(\mathbb{S}^N)} - \frac{1}{2_s^\ast}\int_{\mathbb{S}^N}\vert u\vert^{2_s^\ast} \;dV_g.    
\end{align}
This functional is of class $C^1$ and its derivative is given by
\[
J_U'(u)v = \langle u, v\rangle_{H_g^s(\mathbb{S}^N)} - \int_{\mathbb{S}^N}\vert u\vert^{2_s^\ast-2}u v\; dV_g,\qquad u,v\in H_{g,0}^s(U).
\]
Moreover, we claim that the functional $J_U:H_{g,0}^s(U)\rightarrow\mathbb{R}$ is $G$-invariant. To see this, decompose the functional $J_U$ as $J_U= J_s - L_s$, where
$J_s(u):=\frac{1}{2}\Vert u\Vert_{H_g^s(\mathbb{S}^N)}$ and $L_s(u):=\int_{\mathbb{S}^N}\vert u\vert^{2_s^\ast} \, dV_g$. As it was already noted in Lemma \ref{Lemma:G-invariantFunctional}, $J_s$ is $G$-invariant in $H_g^s(\mathbb{S}^N)$ for every $s\in(0,1]$, so, as $U$ is $G$-invariant, for any $\gamma\in G$, $\gamma(U)= U$ and the functional $J_s$ is also $G$-invariant in $H_{g,0}^s(U)$. To see that $L_s$ is also $G$-invariant, we change variables as follows
\[
L_s(\gamma u) := \int_{\mathbb{S}^N}\vert u\circ\gamma^{-1}\vert^{2_s^\ast}\, dV_g = \int_{U}\vert u\circ\gamma^{-1}\vert^{2_s^\ast}\,dV_g = \int_{U}\vert u\vert^{2_s^\ast}\,dV_g = L_s(u),
\]
and we conclude that $I_s$ is $G$-invariant for any $s\in(0,1]$.

Due to this symmetry and the Principle of Symmetric Criticality \cite{Palais1979}, the critical points of $J_U$ restricted to the space $H_{g,0}^s(U)^G$ correspond to the $G$-invariant solutions to the problem~\eqref{Equation:DirichletProblemSphere}, and to the problem \eqref{Equation:FractionalYamabeSphere} when $U=\mathbb{S}^N$. The nontrivial ones lie on the \emph{Nehari manifold}

\begin{align*}
\mathcal{M}_{U} &:= \left\{u\in H_{g,0}(U)^G\; : \; u\neq 0, J_U'(u)u=0\right\}
=\left\{u\in H_{g,0}(U)^G\; : \; u\neq 0,\ \Vert u\Vert^2_{H_{g}^s(\mathbb{S}^N)} = \int_{U}\vert u\vert^{2_s^\ast}\; dV_g\right\}.
\end{align*}

By Sobolev inequalities,
\[
c_U := \inf_{u\in \mathcal{M}_{U}} J_U(u)>0.
\]
We say $u\in \mathcal{M}_U$ is a least energy $G$-invariant solution to \eqref{Equation:DirichletProblemSphere} if $J(u)=c_U$. We say that it is positive if it is positive in $U$.

\subsection{Fractional Yamabe systems on the sphere} 

Fix $\ell\in\mathbb{N}$ and $\eta_{ij}, a_{ij},b_{ij}\in\mathbb{R}$ such that  $\eta_{ij}=\eta_{ji}<0$, $a_{ij}, b_{ij}>1$, $a_{ij}=b_{ji},$ and $a_{ij}+ b_{ij} = 2_s^\ast$ as in the introduction.
We say that  $\overline{u}=(u_1,\ldots,u_\ell)\in \mathcal{H}_{\mathbb{S}^N}:=\left(H_g^s(\mathbb{S}^N)^G\right)^\ell$ is a $G$-invariant \emph{weak solution} to the system \eqref{Equation:SystemSphere}
if 
\[
\langle u_i, v\rangle_{H_g^s(\mathbb{S}^N)} = \int_{\mathbb{S}^N} \vert u_i\vert^{2_{s}^\ast-2}u_i v\; dV_g + \sum_{j\neq i}^\ell \int_{\mathbb{S}^N} \eta_{ij}b_{ij}\vert u_j\vert^{a_{ij}} \vert u_i\vert^{b_{ij}-2}u_i v\; dV_g\quad 
\text{ for }v\in H_g^s(\mathbb{S}^N),\ i=1,\ldots,\ell.
\]
Let
\[
\Vert \overline{u}\Vert_{\mathcal{H}_{\mathbb{S}^N}}:= \Vert (u_1,\ldots,u_\ell) \Vert_{\mathcal{H}_{\mathbb{S}^N}} := \sqrt{\sum_{i=1}^\ell \Vert u_i\Vert^2_{H_g^s(\mathbb{S}^N)}}
\]
be the norm in $\mathcal{H}_{\mathbb{S}^N}$ induced by the interior product $\langle \overline{u}, \overline{v} \rangle_{\mathcal{H}_{\mathbb{S}^N}} := \sum_{i=1}^\ell \langle u_i, v_i \rangle_{H_g^s(\mathbb{S}^N)}$, 
and consider the functional $\mathcal{J}: \left(H_g^s(\mathbb{S}^N)\right)^\ell\rightarrow\mathbb{R}$, given by
\begin{align}\label{Jdef}
\mathcal{J}(\overline{u}):= \frac{1}{2}\Vert \overline{u} \Vert^2_{\mathcal{H}_{\mathbb{S}^N}} - \frac{1}{2^\ast_{s}}\sum_{i=1}^\ell\int_{\mathbb{S}^N}\vert u_i\vert^{2^\ast_{s}} dV_g 
- \frac{1}{2}\sum_{i\neq j}\int_{\mathbb{S}^N}\eta_{ij}\vert u_j\vert^{a_{ij}}\vert u_i\vert^{ b_{ij}} dV_g.    
\end{align}
Its partial derivatives are 
\[
\partial_i\mathcal{J}(\overline{u})\overline{v}=\langle u_i,v_i\rangle_{H_g^s(\mathbb{S}^N)} - \int_{\mathbb{S}^N} \vert u_i\vert^{2^\ast_{s}-2}u_i v_i\, dV_g - \sum_{i\neq j}\int_{\mathbb{S}^N}\eta_{ij} b_{ij}\vert u_j\vert^{a_{ij}}\vert u_i\vert^{ b_{ij}-2}u_iv_i\, dV_g
\]
for $i=1,\ldots\ell$ and $\overline{u},\overline{v}\in \left(H_g^s(\mathbb{S}^N)\right)^\ell$. 

As before, $\mathcal{J}$ is $G$-invariant and, by the Principle of Symmetric Criticality \cite{Palais1979}, the critical points of $\mathcal{J}$ restricted to the space $\mathcal{H}_{\mathbb{S}^N}$ correspond to the $G$-invariant solutions to the problem \eqref{Equation:SystemSphere}. The fully nontrivial ones (i.e., each component is non trivial) belong to the set
\[
\mathcal{N}_{\mathbb{S}^N}:=\{\overline{u}\in\mathcal{H}_{\mathbb{S}^N}\; | \; u_i\neq 0,  \partial_i\mathcal{J}(\overline{u})u_i=0 \text{ for every }i=1,\ldots,\ell\}.
\]
Observe that $\mathcal{J}(\overline{u})=\frac{s}{N}\sum_{i=1}^\ell \Vert u_i\Vert^2_{H_g^s(\mathbb{S}^N)}$ for every $\overline{u}\in\mathcal{N}_{\mathbb{S}^N}$.

\begin{lemma} \label{lem:away_froM_{d/2}}
	There exists $d_0>0$, independent of $\eta_{ij}$, such that $\min_{i=1,\ldots,\ell}\|u_i\|_{H_g^s(\mathbb{S}^N)}\geq d_0$ if $\overline{u}\in \mathcal{N}_{\mathbb{S}^N}$. 
\end{lemma}

\begin{proof}
	Since $\eta_{ij}<0$  for any $\overline{u}\in\mathcal{N}_{\mathbb{S}^N}$, then the Sobolev inequality gives the existence of a constant $C>0$ such that
	\begin{align*}
		\|u_i\|_{H_g^s(\mathbb{S}^N)}^2\leq \int_{\mathbb{S}^N} |u_i|^{2_{s}^\ast}\leq C\|u_i\|_{H_g^s(\mathbb{S}^N)}^{2_{s}^\ast},\quad  \ i=1,\ldots,\ell.
	\end{align*}
	The result follows from this inequality and from the definition of $\mathcal{N}_{\mathbb{S}^N}$.
\end{proof}

By this lemma, the  set $\mathcal{N}_{\mathbb{S}^N}$ is closed in $\mathcal{H}_{\mathbb{S}^N}$ and 
\begin{equation}\label{Equation:InfimumNehari}
\inf_{\overline{u}\in\mathcal{N}_{\mathbb{S}^N}}\mathcal{J}(\overline{u})>0.
\end{equation}
We say that $\overline{u}\in\mathcal{H}_{\mathbb{S}^N}$ is a \emph{least energy $G$-invariant weak solution of~\eqref{Equation:SystemSphere}} if $\overline{u}$ attains this infimum. 

\subsection{Fractional Yamabe equations in the Euclidean space}

We now introduce the symmetric setting in the Euclidean space. As mentioned after Theorem~\ref{Theorem:MainOptimalPartitionPart1}, via the stereographic projection, we can define an action of $G$ in $\mathbb{R}^N$ as follows: for each $\gamma\in G$, define the map $\widetilde{\gamma}:\mathbb{R}^N\rightarrow\mathbb{R}^N$ given by $\widetilde{\gamma}(x):=(\sigma\circ\gamma^{-1}\circ\sigma^{-1})(x)$, which is well defined except for a single point. Observe that $\widetilde{\gamma}$ is not an isometry of $\mathbb{R}^N,$ but it is a conformal diffeomorphism, that is, a diffeomorphism that preserves angles. 
Accordingly, we recall that $\Omega\subset\mathbb{R}^N$ is $G$-invariant if $\tilde{\gamma}x\in\Omega$ for every $x\in\Omega$ and every $\gamma\in G$.
If $\Omega$ is a $G$-invariant and smooth open subset of $\mathbb{R}^N$, possibly $\Omega=\mathbb{R}^N$, the action of $G$ in $\mathbb{R}^N$ also induces an action in $D^s(\Omega)$ as follows: we say that a function $v\in D^s(\Omega)$ is $G$-invariant if
\[
\vert\det D\widetilde{\gamma}(x)\vert^{1/2_{s}^\ast} v(\widetilde{\gamma}x) = v (x),\qquad \text{for every }x\in\Omega, \gamma\in G.
\]

 Observe that $U\subset\mathbb{S}^N$ is a $G$-invariant and smooth open subset if and only if $\tilde{U}:=\sigma(U)\subset\mathbb{R}^N$ is $G$-invariant. Noting that
\begin{align*}
|\det D\widetilde{\gamma}(x)| = \left(\frac{\psi_s(x)}{\psi_s(\widetilde{\gamma}(x))}\right)^{2_s^*},    
\end{align*}
we also conclude that a function $u:U\to\mathbb{R}$ is $G$-invariant if and only if $\tilde{u}:=\psi_su\circ\sigma^{-1}:\tilde{U}\to\mathbb{R}$ is $G$-invariant. 

Given a $G$-invariant smooth domain $\Omega$ in $\mathbb{R}^N$, the space
\[
D^s(\Omega)^G:=\{v\in D^s(\Omega)\;:\; v \text{ is }G\text{-invariant}\}
\]
coincides with the closure of $C_c^\infty(\Omega)^G$ in $D^s(\Omega)$ and it is infinite dimensional, due to the existence of $G$-invariant partitions of the unity \cite{Palais1961}. We say that a function $v\in D^s(\Omega)^G$ is a \emph{$G$-invariant weak solution} of \eqref{eq:3a:intro}
if 
\[
\langle v,w \rangle = \int_\Omega \vert v\vert^{2_s^\ast-2}vw\, dx\qquad \text{ for any }w\in C_c^\infty(\Omega).
\]

When $\Omega=\mathbb{R}^N$, a function $v\in D^s(\mathbb{R}^N)^G$ is a \emph{$G$-invariant weak solution} to the symmetric fractional Yamabe problem \eqref{fyprn} if
\[
\langle v,w \rangle = \int_{\mathbb{R}^N} \vert v\vert^{2_s^\ast-2}vw\, dx\qquad \text{ for any }w\in C_c^\infty(\mathbb{R}^N).
\]

Consider the energy functional and the Nehari manifold
\begin{align}\label{JOdef}
J_\Omega(v):=\frac{1}{2}\Vert v\Vert^2-\frac{1}{2_s^*}\int_\Omega|v|^{2_s^*},\qquad 
\mathcal{M}_\Omega:=\{v\in D^{s}(\Omega)^G:v\neq 0,\;J_\Omega'(v)v=0\}.
\end{align}

We say that $v\in \mathcal{M}_\Omega$ is a \emph{least energy} $G$-invariant weak solution of \eqref{eq:3a:intro} if $J_\Omega(v)=\inf_{\mathcal{M}_\Omega}J_\Omega.$  By Sobolev inequalities,
\begin{align*}
   c_\Omega:= \inf_{\mathcal{M}_\Omega}J_\Omega>0.
\end{align*}

\subsection{Fractional Yamabe systems in the Euclidean space}

Given $\ell\in\mathbb{N}$ and $\eta_{ij}, a_{ij},b_{ij}\in\mathbb{R}$ such that  $\eta_{ij}=\eta_{ji}<0$, $a_{ij}, b_{ij}>1$, $a_{ij}=b_{ji}$ and $a_{ij}+ b_{ij} = 2_s^\ast$, for each $i,j=1,\ldots,\ell$, we say that  $\overline{v}=(v_1,\ldots,v_\ell)\in \mathcal{H}_{\rn}:=\left(D^s(\rn)^G\right)^\ell$ is a \emph{$G$-invariant weak solution} to the system
\begin{equation}\label{Eq:SystemR^N}
(-\Delta)^s v_i = \vert v_i\vert^{2_{s}^\ast-2}v_i - \sum_{i\neq j} \eta_{ij}\vert v_j\vert^{a_{ij}}\vert v_i\vert^{ b_{ij}-2}v_i\quad \text{in }\mathbb{R}^N,\qquad  i,j=1,\ldots,\ell,
\end{equation}
if 
\[
\langle v_i, \vartheta\rangle = \int_{\rn} \vert v_i\vert^{2_{s}^\ast-2}v_i \vartheta\; dx+ \sum_{j\neq i}^\ell \int_{\rn} \eta_{ij}b_{ij}\vert v_j\vert^{a_{ij}} \vert v_i\vert^{b_{ij}-2}v_i \vartheta\; dx\quad 
\text{ for all }\vartheta\in D^s(\rn),\ i=1,\ldots,\ell.
\]

As in the case of the sphere, let
\[
\Vert \overline{v}\Vert:= \Vert (v_1,\ldots,v_\ell) \Vert := \sqrt{\sum_{i=1}^\ell \Vert v_i\Vert^2}
\]
be the norm in $\mathcal{H}_{\rn}$ induced by the interior product $\langle \overline{v}, \overline{\vartheta} \rangle := \sum_{i=1}^\ell \langle v_i, \vartheta_i \rangle$. Making a slight abuse of notation, we also define the the functional $\mathcal{J}: \left(D^s(\rn)\right)^\ell\rightarrow\mathbb{R}$ as
\[
\mathcal{J}(\overline{v}):= \frac{1}{2}\Vert \overline{v} \Vert^2 - \frac{1}{2^\ast_{s}}\sum_{i=1}^\ell\int_{\rn}\vert v_i\vert^{2^\ast_{s}} dx
- \frac{1}{2}\sum_{i\neq j}\int_{\rn}\eta_{ij}\vert v_j\vert^{a_{ij}}\vert v_i\vert^{ b_{ij}} dx.
\]

This is the same notation we used in \eqref{Jdef}, but it should not cause confusion since its meaning is clear from context.  As before, $\mathcal{J}$ is $G$-invariant and, by the Principle of Symmetric Criticality \cite{Palais1979}, the critical points of $\mathcal{J}$ restricted to the space $\mathcal{H}_{\mathbb{R}^N}$ correspond to the $G$-invariant weak solutions to the problem. The fully nontrivial ones belong to the set
\begin{align}\label{Nrndef}
\mathcal{N}_{\rn}:=\{\overline{v}\in\mathcal{H}_{\rn}\; | \; v_i\neq 0,  \partial_i\mathcal{J}(\overline{v})v_i=0 \text{ for every }i=1,\ldots,\ell\}.
\end{align}
Observe that $\mathcal{J}(\overline{v})=\frac{s}{N}\sum_{i=1}^\ell \Vert v_i\Vert^2$ for every $\overline{v}\in\mathcal{N}_{\rn}$.  We say that $\overline{v}\in \mathcal{N}_{\rn}$ is a least energy solution of \eqref{Eq:SystemR^N} if $\mathcal{J}(\overline{v})=\inf_{\mathcal{N}_{\rn}}\mathcal{J}.$
Arguing as in Lemma~\ref{lem:away_froM_{d/2}}, we also have that
\begin{align}\label{Equation:InfimumNehariRN}
    \inf_{\mathcal{N}_{\rn}}\mathcal{J}>0.
\end{align}

\section{Existence results} \label{Section:SystemsAndSegregation}

In this section, we prove the existence of positive least energy $G$-invariant weak solutions to the symmetric fractional Yamabe problems \eqref{Equation:FractionalYamabeSphere} and
\eqref{fyprn}, to the symmetric Dirichlet boundary problems \eqref{Equation:DirichletProblemSphere} and \eqref{eq:3a:intro}, and to the systems \eqref{Eq:SystemR^N} and \eqref{Equation:SystemSphere}.

\subsection{Equivalence via isometry} \label{Section:Stereographic}

First, we state and prove the equivalences between problems on the sphere and in the Euclidean space. To this end, notice that the isometry 
\begin{align*}
\iota_s : u\mapsto \widetilde{u}:=\psi_s u\circ\sigma^{-1}    
\end{align*}
given in Proposition~\ref{Proposition:IsometricIsomorphism} gives (with a slight abuse of notation) the isometric isomorphisms
\begin{equation}\label{SymmetricIsometryDomains}
\iota_s: H_g^s(\mathbb{S}^N)^G \rightarrow D^s(\mathbb{R}^N)^G
\qquad \text{ and }\qquad 
\iota_s: H_g^s(U)^G \rightarrow D^s(\widetilde{U})^G,
\end{equation}
where $U\subset\mathbb{S}^N$ and $\widetilde{U}:=\sigma(U)\subset\mathbb{R}^N$ are $G$-invariant.

\begin{lemma}\label{Lemma:EquivalentProblemsGinvariant}
Let $s\in(0,1)$, 
\begin{enumerate}[(i)]
\item $u\in H_g^s(\mathbb{S}^N)^G$ is a $G$-invariant weak solution of the fractional Yamabe problem ~\eqref{Equation:FractionalYamabeSphere} if and only if $\widetilde{u}\in D^s(\mathbb{R}^N)^G$ is a $G$-invariant weak solution of the fractional Yamabe problem ~\eqref{fyprn}.
\item For any $G$-invariant smooth domain  $U\subset\mathbb{S}^N$, if $\widetilde{U}:=\sigma(U)$, then $u\in H_{g,0}^s(U)^G$ is a $G$-invariant weak solution of \eqref{Equation:DirichletProblemSphere} if and only if $\widetilde{u}\in D^{s}(\widetilde{U})^G$ is a $G$-invariant weak solution of~\eqref{eq:3a:intro}. 
\item $\overline{u}=(u_1,\ldots,u_\ell)\in \left(H_g^s(\mathbb{S}^N)^G\right)^\ell$ is a $G$-invariant weak solution of~\eqref{Equation:SystemSphere} if and only if $\overline{v}:=(\widetilde{u}_1,\ldots,\widetilde{u}_\ell)\in \left(D^s(\mathbb{R}^N)^G\right)^\ell$ is a $G$-invariant weak solution of~\eqref{Eq:SystemR^N}.
\end{enumerate}
\end{lemma}

\begin{proof}
We only argue $(iii)$, since $(i)$ and $(ii)$ follow similarly. Let $\overline{u}=(u_1,\ldots,u_\ell)\in \left(H_g^s(\mathbb{S}^N)^G\right)^\ell$ and note that
\begin{equation}\label{Equation:EquivalenceCriticalNorms2}
\int_{\mathbb{S}^N} \vert u_j\vert^{a_{ij}}\vert u_i\vert^{ b_{ij}-2}u_i w \, dV_g = \int_{\mathbb{R}^N} \vert \widetilde{u}_j\vert^{a_{ij}}\vert \widetilde{u}_i\vert^{ b_{ij}-2}\widetilde{u}_i \widetilde{w}\, dx \qquad \text{for any }\ w\in  C^\infty(\mathbb{S}^N),
\end{equation}
and
\begin{equation}\label{Equation:EquivalentSystems}
\int_{\mathbb{S}^N} \vert u_i\vert^{2_{s}^\ast-2}u_i w \, dV_g = \int_{\mathbb{R}^N} \vert \widetilde{u}_i\vert^{2_{s}^\ast-2}\widetilde{u}_i \widetilde{w}\, dx \qquad \text{for any }\ w\in  C^\infty(\mathbb{S}^N).
\end{equation}
Hence, $(iii)$ follows from \eqref{Equation:Isometry} together with ~\eqref{Equation:EquivalenceCriticalNorms2}, and~\eqref{Equation:EquivalentSystems}.

%

\end{proof}

\subsection{Existence of symmetric solutions to Yamabe problems}\label{Section:ExistenceDirichlet}

We next prove the existence of least energy $G$-invariant solutions to the problems \eqref{Equation:FractionalYamabeSphere} and \eqref{Equation:DirichletProblemSphere}. The main tool is the following Sobolev embedding for the fractional Sobolev spaces in the presence of symmetries, needed to restore the lack of compactness in critical Sobolev exponent problems. In the following, to emphasize the role of the dimension, we use the notation $2^\ast_{L,s}=\frac{2L}{L-2s}$ to denote the critical Sobolev exponent of dimension $L$ and, as before, $2^\ast_{N,s}=2^*_s$.  Recall that $G=O(m)\times O(n)$ with $m+n=N+1$ and $m,n\geq 2.$

\begin{lemma}\label{Lemma:SymmetricSobolevEmbeddings}
Let $U\subset\mathbb{S}^N$ be a $G$-invariant smooth domain and define $\kappa:=\min_{z\in\mathbb{S}^N}\dim Gz$ such that $\kappa \geq 1$. Then, for any $s\in(0,1)$, the embedding 
\[
H_{g,0}^s(U)^G\hookrightarrow L_g^p(U)
\]
is continuous for $p\in[1,2^\ast_{N-\kappa,s}]$ and compact for $p\in[1,2^\ast_{N-\kappa,s})$. In particular, the embeddings $H_g^s(\mathbb{S}^N)^G\hookrightarrow L_g^{2^*_s}(\mathbb{S}^N)$ and $D^{s}(\R^N)^G\hookrightarrow L^{2^*_s}(\R^N)$ are compact. Moreover, the embedding $H_{g,0}^s(U)^G\hookrightarrow L_g^{2_s^\ast}(U)$ is compact and if $\Omega$ is a smooth $G$-invariant domain in $\mathbb{R}^N$, then also the embedding $D^s(\Omega)^G\hookrightarrow L^{2^\ast_s}(\Omega)$ is compact.

\end{lemma}

\begin{proof}
The case $U = \mathbb{S}^N$ is proven in \cite[Lemma 4.3]{AbreuBarbosaRamirez2022}, following the ideas introduced in \cite{HebeyVaugon1997}. For the case $U\neq\mathbb{S}^N$, as $H_{g,0}^s(U)^G\hookrightarrow H_{g}^s(\mathbb{S}^N)$ is continuous, the composition
\[
H_{g,0}^s(U)^G\hookrightarrow H_{g}^s(\mathbb{S}^N)^G\hookrightarrow L_g^p(\mathbb{S}^N)
\]
is continuous for any $p\in[1,2^\ast_{N-\kappa,s}]$ and compact for any $p\in[1,2^\ast_{N-\kappa,s})$. Since $\kappa\geq 1$ because $m,n\geq 2$, we have that $p=2_{N-\kappa,s}^\ast<2_{N,s}^\ast=2_s^\ast$, where we conclude the desired embeddings for domains in the sphere. The results in the Euclidean space follow from these embeddings and the isometries given in Proposition~\ref{Proposition:IsometricIsomorphism} and in \eqref{SymmetricIsometryDomains}.
\end{proof}

Now, as a direct consequence of Lemma~\ref{Lemma:SymmetricSobolevEmbeddings}, we prove the existence of positive least energy solutions to the fractional Yamabe problem 
 on the sphere \eqref{Equation:FractionalYamabeSphere} and to the Dirichlet boundary problem \eqref{Equation:DirichletProblemSphere}. The last one implies that the optimal $(G,\ell)$-partition problem~\eqref{Equation:OptimalPartitionProblemSphere} is well defined.

\begin{proposition}\label{Proposition:ExistenceDirichlet}
\begin{enumerate}[(i)]
\item The equation~\eqref{Equation:FractionalYamabeSphere} admits, at least, one positive least energy $G$-invariant weak solution.
\item For any $G$-invariant smooth domain $U\subset\mathbb{S}^N$, the Dirichlet boundary problem~\eqref{Equation:DirichletProblemSphere} admits, at least, one positive least energy $G$-invariant weak solution. In particular, the $G$-invariant optimal partition problem \eqref{Equation:OptimalPartitionProblemSphere} is well defined.
\end{enumerate}
\end{proposition}

\begin{proof}
The existence of least energy $G$-invariant weak solutions follows from Lemma~\ref{Lemma:SymmetricSobolevEmbeddings} together with standard variational methods, for instance, adapting the arguments in \cite[Theorem~1.1]{HS21}. The fact that the solution is nonnegative follows directly from the isometric isomorphism in Proposition~\ref{Proposition:IsometricIsomorphism} and the well known inequality
$
\langle v,v \rangle
\geq 
\langle |\overline{v}|,|\overline{u}| \rangle
$, which is an easy direct computation using that $|v|=v^+-v^-,$ where $v^+:=\max\{v,0\}\geq 0$ and $v^-:=\min\{v,0\}\leq 0$. 
The positivity follows from the maximum principle applied, for instance, to the problem in $\mathbb R^N$ (using the stereographic projection).
\end{proof}

An immediate consequence of this proposition together with Lemma~\ref{Lemma:EquivalentProblemsGinvariant} is the following existence result for the problems in the Euclidean space.

\begin{corollary}\label{Corollary:ExistenceDirichletRN}
\begin{enumerate}[(i)]
\item The fractional Yamabe problem ~\eqref{fyprn} admits, at least one positive least energy $G$-invariant weak solution.
\item For any $G$-invariant smooth domain $\Omega\subset\mathbb{R}^N$, the Dirichlet boundary problem~\eqref{eq:3a:intro} admits, at least, one positive least energy $G$-invariant weak solution. In particular, the $G$-invariant optimal partition problem \eqref{OPPRN} is well defined.
\end{enumerate}
\end{corollary}


\subsection{Existence of symmetric solutions to systems}
\label{Section:ExistenceSystems}

We now prove the existence of least energy $G$-invariant weak solutions to  systems   \eqref{Eq:SystemR^N} and \eqref{Equation:SystemSphere}. The main result of this section is the following.

\begin{proposition}\label{Proposition:MainFractionalSystems}
	For any $s\in(0,1)$, system~\eqref{Equation:SystemSphere} admits an unbounded sequence of $G$-invariant fully nontrivial solutions. One of them is positive in each of its components and attains~\eqref{Equation:InfimumNehari}. 
\end{proposition}

Up to minor modifications, the proof of this proposition follows the ideas introduced in \cite{ClappSzulkin2019}, using Lemma~\ref{Lemma:SymmetricSobolevEmbeddings} as in the proof of \cite[Theorem 1.1]{ClappFernandezSaldana2021} to restore compactness. In what follows, we sketch it for the reader's convenience. 

First, there is a projection on $\mathcal{N}_{\mathbb{S}^N}$ from a suitable subset of $\mathcal{H}_{\mathbb{S}^N}$. To define it, given $\overline{u}=(u_1,\ldots,u_\ell)$ and $\overline{t}=(t_1,\ldots,t_\ell)\in(0,\infty)^\ell$, we write
\[
\overline{t}\,\overline{u}:= (t_1u_1,\ldots,t_\ell u_\ell).
\]
Let $\mathcal{S}:=\{u\in H_g^s(\mathbb{S}^N)^G:\|u\|_{H_g^s(\mathbb{S}^N)}=1\}$, define  $\mathcal{T}:=\mathcal S^\ell$ and 
$$\mathcal{U}:=\{\overline{u}\in\mathcal{T}:\overline{t}\,\overline{u}\in\cN_{\mathbb{S}^N}\text{ \ for some \ }\overline{t}\in(0,\infty)^\ell\}.$$

The next result says that this last set is nonempty and that the projection from $\mathcal{U}$ onto $\mathcal{N}_{\mathbb{S}^N}$ defines a homeomorphism. The proof is  exactly the same as the proof of \cite[Proposition 3.1]{ClappSzulkin2019}.

\begin{lemma} \label{lem:U}
	\begin{itemize}
		\item[$(i)$] Let $\overline{u}\in\mathcal{T}$. If there exists $\overline{t}_{\overline{u}}\in(0,\infty)^\ell$ such that $\overline{t}_{\overline{u}}\overline{u}\in\cN$, then $\overline{t}_{\overline{u}}$ is unique and satisfies
		$$\mathcal{J}(\overline{t}_{\overline{u}}\overline{u})=\max_{\overline{t}\in(0,\infty)^\ell}\mathcal{J}(\overline{t}\,\overline{u}).$$
		\item[$(ii)$] $\mathcal{U}$ is a nonempty open subset of $\mathcal{T}$, and the map $\mathcal{U}\to(0,\infty)^\ell$ given by $\overline{u}\mapsto\overline{t}_{\overline{u}}$ is continuous.
		\item[$(iii)$] The map $\rho:\mathcal{U}\to \cN$ given by $\overline{u}\mapsto\overline{t}_{\overline{u}}\overline{u}$ is a homeomorphism.
		\item[$(iv)$] If $(\overline{u}_n)$ is a sequence in $\mathcal{U}$ and $\overline{u}_n\to\overline{u}\in\partial\mathcal{U}$, then $|\overline{t}_{\overline{u}_n}|\to\infty$.
	\end{itemize}
\end{lemma}
\medskip

The main difficulty to perform variational methods on $\mathcal{N}_{\mathbb{S}^N}$, is that it is not clear if this set is a differentiable manifold where we can apply the negative gradient flow of $\mathcal{J}$. However, by the previous lemma, the set $\mathcal{U}$ is an open subset of $\mathcal{T}$, which is a differentiable manifold and, therefore, $\mathcal{U}$ is also a differentiable manifold which is homeomorphic to the set $\mathcal{N}_{\mathbb{S}^N}$. In what follows, we define a functional $\Psi$ which translates the variational structure encoded in $\mathcal{J}$ and $\mathcal{N}_{\mathbb{S}^N}$ to $\mathcal{U}$, where we can apply the negative gradient flow. More precisely,  Palais-Smale sequences and critical points for $\mathcal{J}$ in $\mathcal{N}_{\mathbb{S}^N}$ induce Palais-Smale sequences and critical points for $\Psi$ in $\mathcal{U}$ and vice versa. 

Define $\Psi:\mathcal{U}\to\mathbb{R}$ as 
\begin{equation*}
	\Psi(\overline{u}): = \mathcal{J}(\overline{t}_{\overline{u}}\overline{u}).
\end{equation*}

According to item (ii) of Lemma~\ref{lem:U}, $\mathcal{U}$ is a Hilbert manifold. When $\Psi$ is differentiable at $\overline{u}$, we write $\|\Psi'(\overline{u})\|_{\bigstar}$ for the norm of $\Psi'(\overline{u})$ in the cotangent space $\mathrm{T}_{\overline{u}}^*(\mathcal{U})$ to $\mathcal{U}$ at $\overline{u}$, i.e.,
\[
\|\Psi'(\overline{u})\|_{\bigstar}:=\sup\limits_{\substack{\overline{v}\in\mathrm{T}_{\overline{u}}(\mathcal{U}) \\\overline{v}\neq 0}}\frac{|\Psi'(\overline{u})\overline{v}|}{\|\overline{v}\|_{\mathcal{H}_{\mathbb{S}^N}}},
\]
where $\mathrm{T}_{\overline{u}}(\mathcal{U})$ is the tangent space to $\mathcal{U}$ at $\overline{u}$ \cite[Section III.4]{LangBook}.

Recall that a sequence $(\overline{u}_n)$ in $\mathcal{U}$ is called a $(PS)_c$\emph{-sequence for} $\Psi$ if $\Psi(\overline{u}_n)\to c$ and $\|\Psi'(\overline{u}_n)\|_{\bigstar}\to 0$, and $\Psi$ is said to satisfy the $(PS)_c$\emph{-condition} if every such sequence has a convergent subsequence. Similarly, a $(PS)_c$\emph{-sequence for} $\mathcal{J}$ is a sequence $(\overline{u}_n)$ in $\mathcal{H}_{\mathbb{S}^N}$ such that $\mathcal{J}(\overline{u}_n)\to 0$ and $\|\mathcal{J}'(\overline{u}_n)\|_{\mathcal{H}_{\mathbb{S}^N}'}\to 0$, and $\mathcal{J}$ satisfies the $(PS)_c$\emph{-condition} if any such sequence has a convergent subsequence.   Here, as usual, $\mathcal{H}_{\mathbb{S}^N}'$ denotes the dual space of $\mathcal{H}_{\mathbb{S}^N}$. The next lemma can be argued exactly as in \cite[Theorem 3.3]{ClappSzulkin2019}, and we omit its proof.

\begin{lemma}\label{Lemma:Psi}
	\begin{itemize}
		\item[$(i)$] $\Psi\in  C^1(\mathcal{U})$ and its derivative is given by
		\begin{equation*}
			\Psi'(\overline{u})\overline{v} = \mathcal{J}'(\overline{t}_{\overline{u}}\overline{u})[\overline{t}_{\overline{u}}\overline{v}] \quad \text{for all } \overline{u}\in\mathcal{U} \text{ and }\overline{v}\in \mathrm{T}_{\overline{u}}(\mathcal{U}).
		\end{equation*}
		Moreover, there exists $d_1>0$ such that
		$$d_1\min_i\{t_{u,i}\}  \|\mathcal{J}'(\overline{t}_{\overline{u}}\overline{u})\|_{\mathcal{H}_{\mathbb{S}^N}'}\leq\|\Psi'(\overline{u})\|_{\bigstar}\leq \max_i\{t_{u,i}\} \|\mathcal{J}'(\overline{t}_{\overline{u}}\overline{u})\|_{\mathcal{H}_{\mathbb{S}^N}'}\quad \text{for all } \overline{u}\in\mathcal{U}.$$
		\item[$(ii)$] If $(\overline{u}_n)$ is a $(PS)_c$-sequence for $\Psi$ in $\mathcal{U}$, then $(\overline{t}_{\overline{u}_n}\overline{u}_n)$ is a $(PS)_c$-sequence for $\mathcal{J}$ in $\mathcal{H}_{\mathbb{S}^N}$.
		\item[$(iii)$] $\overline{u}$ is a critical point of $\Psi$ if and only if $\overline{t}_{\overline{u}}\overline{u}$ is a fully nontrivial critical point of $\mathcal{J}$.
		\item[$(iv)$] If $(\overline{u}_n)$ is a sequence in $\mathcal{U}$ and $\overline{u}_n\to\overline{u}\in\partial\mathcal{U}$, then $|\Psi(\overline{u}_n)|\to\infty$.
		\item[$(v)$]$\overline{u}\in\mathcal{U}$ if and only if $-\overline{u}\in\mathcal{U}$, and $\Psi(\overline{u})=\Psi(-\overline{u})$.
	\end{itemize}
\end{lemma}

Next, we use Lemma~\ref{Lemma:SymmetricSobolevEmbeddings} to restore compactness of the Palais-Smale sequences.

\begin{lemma}
For every $c\in\mathbb{R}$,	$\Psi$ satisfies the $(PS)_c$-condition.
\end{lemma}

\begin{proof}
Let $\rho$ be as in Lemma \ref{lem:U}. 	First, observe that if $(\overline{v}_n)$ is a $(PS)_c$-sequence for $\mathcal{J}$ in $\mathcal{N}_{\mathbb{S}^N}$, there exists $C>0$ such that
	\[
	\frac{s}{N}\Vert \overline{v}_n \Vert^2_{\mathcal{H}_{\mathbb{S}^N}} = \mathcal{J}(\overline{v}_n) - \frac{1}{2_{s}^\ast} \mathcal{J}'(\overline{v}_n)\overline{v}_n \leq C(1 + \Vert \overline{v}_n\Vert_{\mathcal{H}_{\mathbb{S}^N}}),
	\]
	and it follows that $(\overline{v}_n)$ is bounded in $\mathcal{H}_{\mathbb{S}^N}$. Now, let $(\overline{u}_n)$ be a $(PS)_c$-sequence for $\Psi$ in $\mathcal{U}$. By Lemma~\ref{Lemma:Psi}, the sequence $\overline{v}_n:=\rho(\overline{u})\in\mathcal{N}_{\mathbb{S}^N}$ is a $(PS)_c$-sequence for $\mathcal{J}$ and it is bounded by the above claim. A standard argument using Lemma~\ref{Lemma:SymmetricSobolevEmbeddings} as in \cite[Proposition 3.6]{ClappPistoia2018}, shows that $(\overline{v}_n)$ contains a convergent subsequence, converging to some  $\overline{v}\in\mathcal{H}_{\mathbb{S}^N}$. As $\overline{v}_n\in\mathcal{N}_{\mathbb{S}^N}$ for every $n\in\mathbb{N}$ and as $\mathcal{N}_{\mathbb{S}^N}$ is closed in $\mathcal{H}_{\mathbb{S}^N}$, it follows that $\overline{v}\in\mathcal{N}_{\mathbb{S}^N}$. Finally, since $\rho$ is a homeomorphism between $\mathcal{N}_{\mathbb{S}^N}$ and $\mathcal{U}$, this yields that $\overline{u}_n$ converges to $\rho^{-1}(\overline{v})$ in a subsequence, and $\Psi$ satisfies the $(PS)_c$-condition.
\end{proof}

Given a nonempty subset $\mathcal{V}$ of $\mathcal{T}$ such that $\overline{u}\in\mathcal{V}$ if and only if $-\overline{u}\in\mathcal{V}$, the \emph{genus of $\mathcal{V}$}, denoted by $\mathrm{genus}(\mathcal{V})$, is the smallest integer $k\geq 1$ such that there exists an odd continuous function $\mathcal{V}\rightarrow\mathbb{S}^{k-1}$ into the unit sphere $\mathbb{S}^{k-1}$ in $\mathbb{R}^k$. If no such $k$ exists, we define $\mathrm{genus}(\mathcal{V})=\infty$; finally, we set $\mathrm{genus}(\emptyset)=0$.

\begin{lemma}
	$\mathrm{genus}(\mathcal{U})=\infty$.
\end{lemma}

\begin{proof}
	Since $\dim Gz\geq 1$ for any $z\in\mathbb{S}^N$, the existence of $G$-invariant partitions of the unity (see \cite{Palais1961}), guarantees the existence of arbitrarily large number of positive $G$-invariant functions in $  C^\infty(\mathbb{S}^N)$ with mutually disjoint supports. Then, arguing as in \cite[Lemma 4.5]{ClappSzulkin2019}, one shows that $\mathrm{genus}(\mathcal{U})=\infty$, as we wanted.
\end{proof}

We next prove that the least energy fully nontrivial $G$-invariant solutions can be taken to be nonnegative in each of its components.

\begin{lemma} \label{Lemma:PositiveLeastEnergySolution}
For $s\in(0,1)$, if $\inf_{\mathcal{N}_{\mathbb{S}^N}}\mathcal{J}=\mathcal{J}(\overline{u})$ for some $\overline{u}\in\mathcal{N}_{\mathbb{S}^N}$, then  $\vert \overline{u}\vert:= (\vert u_1\vert,\ldots,\vert u_\ell\vert)\in\mathcal{N}_{\mathbb{S}^N}$ is a fully nontrivial $G$-invariant least energy solution to the problem~\eqref{Equation:SystemSphere}.
\end{lemma}
\begin{proof}
The proof follows the ideas in Proposition ~\ref{Proposition:ExistenceDirichlet} and we omit it.
\end{proof}

\begin{proof}[Proof of Proposition~\ref{Proposition:MainFractionalSystems}]
	Lemma~\ref{Lemma:Psi} $(iv)$ implies that $\mathcal{U}$ is positively invariant under the negative pseudogradient flow of $\Psi$, so the usual deformation lemma holds true for $\Psi$, see e.g. \cite[Section II.3]{StruweBook} or \cite[Section 5.3]{WillemBook}. As $\Psi$ satisfies the $(PS)_c$-condition for every $c\in\mathbb{R}$, standard variational arguments show that $\Psi$ attains its minimum on $\mathcal{U}$ at some $\overline{u}$. By Lemma~\ref{Lemma:Psi}$(iii)$ and the Principle of Symmetric Criticality \cite{Palais1979}, $\overline{t}_{\overline{u}}\overline{u}$ is a $G$-invariant least energy fully nontrivial solution for the system~\eqref{Equation:SystemSphere}. Moreover, as $\Psi$ is even and $\mathrm{genus}(\mathcal{U})=\infty$, arguing as in the proof of Theorem 3.4 (c) in \cite{ClappSzulkin2019}, it follows that  $\Psi$ has an unbounded sequence of critical points. Using Lemma~\ref{Lemma:Psi} (iii), and the fact that $\Psi(\overline{u})=\mathcal{J}(\overline{t}_{\overline{u}}\overline{u})=\frac{s}{N}\|\overline{t}_{\overline{u}}\overline{u}\|^2$, the system~\eqref{Equation:SystemSphere} has an unbounded sequence of fully nontrivial $G$-invariant solutions.
\end{proof}

As a direct consequence of Proposition~\ref{Proposition:MainFractionalSystems} and Lemma \ref{Lemma:EquivalentProblemsGinvariant}, we have the following existence result for the Euclidean system.

\begin{corollary}\label{Corollary:ExistenceSystemeuclidean}
For any $s\in(0,1)$, the system~\eqref{Eq:SystemR^N} admits an unbounded sequence of $G$-invariant fully nontrivial solutions. One of them is positive in each of its components and attains~\eqref{Equation:InfimumNehariRN}. 
\end{corollary}


\section{Optimal partitions} \label{sec:euclidean}

By Corollary~\ref{Corollary:ExistenceDirichletRN}, the optimal $(G,\ell)$-partition problem in $\mathbb{R}^N$,
\begin{equation} \label{eq:4a}
\inf_{\{\Omega_1,\ldots,\Omega_\ell\}\in\mathfrak{P}_\ell}\;\sum_{i=1}^\ell c_{\Omega_i},\qquad\text{where }c_{\Omega_i}:=\inf_{\mathcal{M}_{\Omega_i}}J_{\Omega_i},
\end{equation}
is well defined, where $\mathfrak{P}_\ell$ is given in~\eqref{Prn} and, by Lemma
~\ref{Lemma:EquivalentProblemsGinvariant} and Proposition ~\ref{Proposition:ExistenceDirichlet}, it is equivalent to the optimal $(G,\ell)$-partition problem on the sphere~\eqref{Equation:OptimalPartitionProblemSphere}.  More precisely, setting $U_i:=\sigma^{-1}(\Omega_i)$, where $\sigma$ is the stereographic projection, we have that $\{U_1,\ldots,U_\ell\}$ solves the optimal $(G,\ell)$-partition problem~\eqref{Equation:OptimalPartitionProblemSphere} on $\mathbb{S}^N$ if and only if $\{\Omega_1,\ldots,\Omega_\ell\}$ solves the optimal $(G,\ell)$-partition problem~\eqref{eq:4a} in $\mathbb{R}^N$.

Note that, if $\{\Omega_1,\ldots,\Omega_\ell\}\in\mathfrak{P}_\ell$ and $v_i\in\mathcal{M}_{\Omega_i}$, then, since $v_iv_j=0$ for $i\neq j$, we have that $(v_1,\ldots,v_\ell)\in\mathcal{N}_{\rn}$ (see \eqref{Nrndef}) and $\mathcal{J}(v_1,\ldots,v_\ell)=J_{\Omega_1}(v_1)+\cdots+J_{\Omega_\ell}(v_\ell)$. Therefore, $\inf_{\mathcal{N}_{\rn}}\mathcal{J}\leq c_{\Omega_1}+\cdots+c_{\Omega_\ell}$ and, consequently,
\begin{equation} \label{eq:comparison}
\inf_{\mathcal{N}_{\rn}}\mathcal{J}\leq \inf_{\{\Omega_1,\ldots,\Omega_\ell\}\in\mathfrak{P}_\ell}\;\sum_{i=1}^\ell c_{\Omega_i}.
\end{equation}

Theorem~\ref{Theorem:MainOptimalPartition} can be restated in $\mathbb{R}^N$ as follows.

\begin{theorem} \label{thm:main_RN}
For each $i,j=1,\ldots,\ell$, $i\ne j$, let $(\eta_{ij,k})$ be a sequence of negative numbers such that $\eta_{ij,k}\to -\infty$ as $k\to\infty$, and  let $v_{k}=(v_{k,1},\ldots,v_{k,\ell})$ be a positive least energy fully nontrivial $G$-invariant solution to the system~\eqref{Eq:SystemR^N} with $\eta_{ij}=\eta_{ij,k}$. Then, after passing to a subsequence, we have that
\begin{itemize}
\item[$(a)$]$v_{k,i}\to v_{\infty,i}$ strongly in $D^{s}(\mathbb{R}^N)$,\, $v_{\infty,i}\geq 0$,\, $v_{\infty,i}$ is continuous and $v_{\infty,i}|_{\Omega_i}$ is a least energy solution to the problem~\eqref{eq:3a:intro} in $\Omega_i:=\{x\in\mathbb{R}^{N}:v_{\infty,i}(x)>0\}$, for each $i=1,\ldots,\ell$.
\item[$(b)$]$\{\Omega_1,\ldots,\Omega_\ell\}\in\mathcal{P}_\ell$ and it solves the optimal $(G,\ell)$-partition problem~\eqref{eq:4a} in $\mathbb{R}^{N}$.
\item[$(c)$]$\Omega_1,\ldots,\Omega_\ell$ are smooth and connected, $\overline{\Omega_1\cup\cdots\cup \Omega_\ell}=\mathbb{R}^{N}$ and, after reordering, we have that $\Omega_1,\ldots,\Omega_{\ell-1}$ are bounded, $\Omega_\ell$ is unbounded,
\begin{itemize}
\item[$(c_1)$] $\Omega_1\cong\mathbb{S}^{m-1}\times \mathbb{B}^{n}$,\quad $\Omega_i\cong\mathbb{S}^{m-1}\times\mathbb{S}^{n-1}\times(0,1)$ if  $i=2,\ldots,\ell-1$, and\quad $\Omega_\ell\cong\mathbb{R}^m\times \mathbb{S}^{n-1}$,
\item[$(c_2)$] $\overline{\Omega}_i\cap \overline{\Omega}_{i+1}\cong\mathbb{S}^{m-1}\times\mathbb{S}^{n-1}$ and\quad $\overline{\Omega}_i\cap \overline{\Omega}_j=\emptyset$\, if\, $|j-i|\geq 2$.
\end{itemize}
\end{itemize}
\end{theorem}

Theorem~\ref{thm:main_RN} follows from the next two theorems, which are of independent interest. Let
$$\widetilde{q}:=q\circ\sigma^{-1}:\mathbb{R}^N\to[0,\pi],$$
where $\sigma$ is the stereographic projection and $q$ is the $G$-orbit map of $\S^N$ defined as
\begin{align*}
    q(z_1,z_1):=\arccos(|z_1|^2-|z_2|^2),\qquad z_1\in \R^m,\ z_2\in \R^n.
\end{align*}
Writing $\rn=\r^m\times\r^{n-1}$, it is easy to see that $\widetilde q^{\,-1}(0) = \mathbb{S}^{m-1}\times\{0\}$ and $\widetilde q^{\,-1}(\pi) = \{0\}\times \r^{n-1}$.

\begin{theorem} \label{thm:partition}
Let $\{\Theta_1,\ldots,\Theta_\ell\}\in\mathfrak{P}_\ell$ be a solution to the optimal $(G,\ell)$-partition problem~\eqref{eq:4a}. Then, the following statements hold true.
\begin{itemize}
\item[$(i)$] There exist $a_1,\ldots,a_{\ell-1}\in(0,\pi)$ such that
$$(0,\pi)\smallsetminus\bigcup_{i=1}^\ell\widetilde{q}\,(\Theta_i)=\{a_1,\ldots,a_{\ell-1}\}.$$
Therefore, after reordering,
\begin{align*}
\Omega_1&:=\Theta_1\cup(\mathbb{S}^{m-1}\times\{0\})=\widetilde{q}\,^{-1}[0,a_1),\\
\Omega_i&:=\Theta_i=\widetilde{q}\,^{-1}(a_{i-1},a_i)\qquad\text{if }\; i=2,\ldots,\ell-1,\\
\Omega_\ell&:=\Theta_\ell\cup(\{0\}\times\mathbb{R}^{n-1})=\widetilde{q}\,^{-1}(a_{\ell-1},\pi].
\end{align*}
\item[$(ii)$]  $\Omega_1,\ldots,\Omega_\ell$ are smooth and connected, they satisfy $(c_1)$ and $(c_2)$ of \emph{Theorem}~\ref{thm:main_RN}, $\Omega_1,\ldots,\Omega_{\ell-1}$ are bounded, $\Omega_\ell$ is unbounded, $\overline{\Omega_1\cup\cdots\cup \Omega_\ell}=\mathbb{R}^{N}$, and $\{\Omega_1,\ldots,\Omega_\ell\}\in\mathfrak{P}_\ell$ is a solution to the optimal $(G,\ell)$-partition problem~\eqref{eq:4a}.
\end{itemize}
\end{theorem}

\begin{proof}
$(i):$ Note that Lemma~\ref{Lemma:SymmetricSobolevEmbeddings} implies, by a standard argument, that $c_{\Omega}:=\inf_{\mathcal{M}_{\Omega}}J_{\Omega}$ is attained for any $G$-invariant smooth open subset $\Omega$ of $\mathbb{R}^N$ and we may assume the minimizer is strictly positive in $\Omega$.

Let $a,b,c\in(0,\pi)$ with $a<b<c$, and set $\Lambda_1:=\widetilde{q}\,^{-1}(a,b)$,\; $\Lambda_2:=\widetilde{q}\,^{-1}(b,c)$,\; $\Lambda=\widetilde{q}\,^{-1}(a,c)$. Then,
$c_{\Lambda} \leq\min\{c_{\Lambda_1},c_{\Lambda_2}\},$ because $\Lambda_i\subset \Lambda$ for $i=1,2$.  We claim that $c_{\Lambda} <\min\{c_{\Lambda_1},c_{\Lambda_2}\}$.  Indeed, if 
$c_{\Lambda} = \min\{c_{\Lambda_1},c_{\Lambda_2}\}$ then, taking a least energy $G$-invariant solution to~\eqref{eq:3a:intro} in $\Lambda_1$ and extending it by 0 in $\Lambda\backslash \Lambda_1$ we obtain a nontrivial least energy $G$-invariant solution $u$ to~\eqref{eq:3a:intro} in $\Lambda$; but this would contradict the unique continuation principle for the fractional Laplacian (see, for instance, \cite[Theorem 1.4]{fall2014unique}).   Therefore, if $\{\Theta_1,\ldots,\Theta_\ell\}\in\mathfrak{P}_\ell$ is a solution to the optimal $(G,\ell)$-partition problem~\eqref{eq:4a}, then $(0,\pi)\smallsetminus\bigcup_{i=1}^\ell\widetilde{q}\,(\Theta)$ must consist of precisely $\ell-1$ points.

$(ii):$ Clearly, $\Omega_1,\ldots,\Omega_\ell$ are smooth and connected, they satisfy $(c_1)$ and $(c_2)$ of Theorem~\ref{thm:main_RN}, $\Omega_1,\ldots,\Omega_{\ell-1}$ are bounded, $\Omega_\ell$ is unbounded, $\mathbb{R}^{N}=\overline{\Omega_1\cup\cdots\cup \Omega_\ell}$, and $\{\Omega_1,\ldots,\Omega_\ell\}\in\mathfrak{P}_\ell$.

As $\Theta_i\subset \Omega_i$ we have that $c_{\Omega_i}\leq c_{\Theta_i}$ for all $i$. So, since $\{\Theta_1,\ldots,\Theta_\ell\}$ is an optimal partition, we conclude that $\{\Omega_1,\ldots,\Omega_\ell\}$ is also an optimal partition.
\end{proof}

\begin{theorem} \label{thm:phase_separation}
For each $i,j=1,\ldots,\ell$, $i\neq j$, let $(\eta_{ij,k})$ be a sequence of negative numbers such that $\eta_{ij,k}\to -\infty$ as $k\to\infty$, and  let $v_{k}=(v_{k,1},\ldots,v_{k,\ell})$ be a positive least energy fully nontrivial $G$-invariant solution to the system~\eqref{Eq:SystemR^N} with $\lambda_{ij}=\eta_{ij,k}$. Then, after passing to a subsequence, we have that
\begin{itemize}
\item[$(a)$]$v_{k,i}\to v_{\infty,i}$ strongly in $D^{s}(\mathbb{R}^N)$,\, $v_{\infty,i}\geq 0$,\, $v_{\infty,i}$ is continuous in $\mathbb{R}^N$ and $v_{\infty,i}|_{\Omega_i}$ is a least energy solution to the problem~\eqref{eq:3a:intro} in $\Omega_i:=\{x\in\mathbb{R}^{N}:v_{\infty,i}(x)>0\}$, for each $i=1,\ldots,\ell$.
\item[$(b)$]$\{\Omega_1,\ldots,\Omega_\ell\}\in\mathcal{P}_\ell$ and it solves the optimal $(G,\ell)$-partition problem~\eqref{eq:4a}.
\end{itemize}
\end{theorem}

\begin{proof}
To highlight the role of $\eta_{ij,k}$, we write $\mathcal{J}_k$ and $\mathcal{N}_k$ for the functional $\mathcal{J}$ and the set $\mathcal{N}_{\rn}$ associated to the system~\eqref{Eq:SystemR^N} with $\eta_{ij}=\eta_{ij,k}$. By assumption,
$$c_k:= \inf_{\mathcal{N}_k} \mathcal{J}_k =\mathcal{J}_k(v_k)=\frac{s}{N}\sum_{i=1}^\ell\|v_{k,i}\|^2.$$
We define
\begin{align*}
\mathcal{N}_0:=\{(v_1,\ldots,v_\ell)\in\mathcal{H}_{\rn}:\,&v_i\neq 0,\;\|v_i\|^2=|v_i|_{2_s^*}^{2_s^*},  \text{ and }v_iv_j=0\text{ a.e. in }\mathbb{R}^N \text{ if }i\neq j\}.
\end{align*}
Then, $\mathcal{N}_0\subset\mathcal{N}_k$ for all $k\in\mathbb{N}$ and, consequently, 
$$0<c_k\leq c_0:=\inf\left\{\frac{s}{N}\sum_{i=1}^\ell\|v_i\|^2:(v_1,\ldots,v_\ell)\in\mathcal{N}_0\right\}<\infty.$$
So, after passing to a subsequence, using Lemma~\ref{Lemma:SymmetricSobolevEmbeddings} we get that $v_{k,i} \rightharpoonup v_{\infty,i}$ weakly in $D_{0}^{s,2}(\mathbb{R}^N)^G$, $v_{k,i} \to v_{\infty,i}$ strongly in $L^{2_s^*}(\mathbb{R}^N)$ and $v_{k,i} \to v_{\infty,i}$ a.e. in $\mathbb{R}^N$, for each $i=1,\ldots,\ell$. Hence, $v_{\infty,i} \geq 0$. Moreover, as $\partial_i\mathcal{J}_k(v_k)[v_{k,i}]=0$, we have that, for each $j\neq i$,
\begin{align*}
0&\leq\int_{\mathbb{R}^N}\eta_{ij,k}|v_{k,j}|^{a_{ij}}|v_{k,i}|^{b_{ij}}\leq \frac{|v_{k,i}|^{2_s^*}_{2_s^*}}{-\eta_{ij,k}}\leq \frac{C}{-\eta_{ij,k}}.
\end{align*}
Then, Fatou's lemma yields  that
$$0 \leq \int_{\mathbb{R}^N}|v_{\infty,j}|^{a_{ij}}|v_{\infty,i}|^{b_{ij}} \leq \liminf_{k \to \infty} \int_{\mathbb{R}^N}|v_{k,j}|^{a_{ij}}|v_{k,i}|^{b_{ij}} = 0.$$
Hence, $v_{\infty,j} v_{\infty,i} = 0$ a.e. in $\mathbb{R}^N$. On the other hand, as shown in \cite[Proposition 3.1]{ClappSzulkin2019}, using Sobolev's inequality we see that
$$0<d_0 \leq \|v_{k,i}\|^2 \leq |v_{k,i}|_{2_s^*}^{2_s^*}\qquad\text{for all }k\in\mathbb{N},\;i=1,\ldots,\ell.$$
So, as $v_{k,i} \to v_{\infty,i}$ strongly in $L^{2_s^*}(\mathbb{R}^N)$, we conclude that $v_{\infty,i}\neq 0$. And, as $v_{k,i} \rightharpoonup v_{\infty,i}$ weakly in $D^{s}(\mathbb{R}^N)$, we get that
\begin{equation} \label{eq:comparison2}
\|v_{\infty,i}\|^2 \leq |v_{\infty,i}|_{2_s^*}^{2_s^*}\qquad\text{for all }i=1,\ldots,\ell.
\end{equation}
Since $v_{\infty,i}\neq 0$, there is a unique $t_i\in(0,\infty)$ such that $\|t_iv_{\infty,i}\|^2 = |t_iv_{\infty,i}|_{2_s^*}^{2_s^*}$. Then, 
\begin{align*}
    (t_1v_{\infty,1},\ldots,t_\ell v_{\infty,\ell})\in \mathcal{N}_0.
\end{align*}
The inequality~\eqref{eq:comparison2} implies that $t_i\in (0,1]$. Therefore,
\begin{align*}
c_0 &\leq \frac{s}{N}\sum_{i=1}^\ell\|t_iv_{\infty,i}\|^2 \leq \frac{s}{N}\sum_{i=1}^\ell\|v_{\infty,i}\|^2
\leq \frac{s}{N}\liminf_{k\to\infty}\sum_{i=1}^\ell\|v_{k,i}\|^2=\liminf_{k\to\infty} c_k \leq c_0.
\end{align*}
Hence, $v_{k,i} \to v_{\infty,i}$ strongly in $D^{s}(\mathbb{R}^N)^G$,\; $t_i=1$, yielding 
\begin{equation}\label{eq:limit}
\|v_{\infty,i}\|^2 = |v_{\infty,i}|_{2_s^*}^{2_s^*},\qquad\text{and}\qquad\frac{s}{N}\sum_{i=1}^\ell\|v_{\infty,i}\|^2 = \lim_{k\to\infty} c_k.
\end{equation}

Set $Y_1:=\mathbb{S}^{m-1}\times\{0\}$, $Y_2:=\{0\}\times\mathbb{R}^{n-1}$, and $Y:=Y_1\cup Y_2$. Proposition~\ref{Proposition:Regularity} (which was proved in Section~\ref{regularity} above), together with Proposition~\ref{Proposition:IsometricIsomorphism}, imply that $v_{\infty,i}|_{\mathbb{R}^N\smallsetminus Y}$ is continuous. Consequently, 
\begin{align*}
\Theta_i:=\{x\in\mathbb{R}^{N}\smallsetminus Y:v_{\infty,i}(x)>0\}    
\end{align*}
is $G$-invariant and open in $\mathbb{R}^N$.  Since $v_{\infty,i}v_{\infty,j}=0$ if $i\neq j$, we have that $\Theta_i\cap\Theta_j=\emptyset$ if $i\neq j$.  Set $\Omega_i:=\operatorname{int}(\overline{\Theta_i})$.  Then $\Omega_i$ is a nonempty, $G-$invariant, open, smooth, $\Omega_i\cap\Omega_j=\emptyset$ if $i\neq j$, and $u_{\infty,i}(x)=0$ in $\R^N\backslash \Omega_i$.  By Lemma~\ref{A}, $u_{\infty,i}\in D^{s}(\Omega_i)^G$ and, by~\eqref{eq:limit}, $u_{\infty,i}\in \mathcal{M}_{\Omega_i}$ and 
\begin{equation*}
\sum_{i=1}^\ell c_{\Theta_i}\leq\frac{s}{N}\sum_{i=1}^\ell\|v_{\infty,i}\|^2 = \lim_{k\to\infty} c_k \leq \inf_{(\Phi_1,\ldots,\Phi_\ell)\in\mathfrak{P}_\ell}\;\sum_{i=1}^\ell c_{\Phi_i}^G.
\end{equation*}
This shows that $\{\Theta_1,\ldots,\Theta_\ell\}$ solves the optimal $(G,\ell)$-partition problem~\eqref{eq:4a}. 
 
Reordering this partition as indicated in Theorem~\ref{thm:partition}, and setting $\Omega_1:=\Theta_1\cup Y_1$, $\Omega_\ell:=\Theta_\ell\cup Y_2$ and $\Omega_i:=\Theta_i$ if $i\neq 1,\ell$, we have that $\{\Omega_1,\ldots,\Omega_\ell\}\in\mathfrak{P}_\ell$ and $c_{\Omega_i}=c_{\Theta_i}$. As $v_{\infty,i}|_{\Omega_i}\in\mathcal{M}_{\Omega_i}$ and $J_{\Omega_i}(v_{\infty,i}|_{\Omega_i}) = c_{\Omega_i}$, the function $v_{\infty,i}|_{\Omega_i}$ solves problem~\eqref{eq:3a:intro} in $\Omega_i$. Since $\Omega_i$ is smooth by Theorem~\ref{thm:partition}, we have that $v_{\infty,i}$ is continuous in $\mathbb{R}^N$ and $\Omega_i=\{x\in\mathbb{R}^{N}:v_{\infty,i}(x)>0\}$. This concludes the proof. 
\end{proof} \medskip

\begin{proof}[Proof of Theorem~\ref{thm:main_RN}]
This result follows immediately from Theorems~\ref{thm:partition} and~\ref{thm:phase_separation}.
\end{proof}

\begin{proof}[Proof of Theorem~\ref{Theorem:MainOptimalPartition}]
It is a direct consequence of Theorem~\ref{thm:main_RN}, together with Lemma 
\ref{Lemma:EquivalentProblemsGinvariant} and the isometry~\eqref{SymmetricIsometryDomains}.
\end{proof}

\begin{proof}[Proof of Theorem~\ref{Theorem:MainOptimalPartitionPart1}]
  This follows from Theorem~\ref{Theorem:MainOptimalPartition} and Proposition \ref{Proposition:MainFractionalSystems}.
\end{proof}

\begin{remark}\label{schs}
In the local setting $s=1$, it is possible to use the limit profiles $v_{\infty,i}$ to build a solution of the Yamabe problem in $\rn$, namely, the function $v:=\sum_{i=1}^\ell (-1)^{i-1}v_{i,\infty}$ is a solution to $-\Delta v = |v|^{2^*-2}v$ in $\rn,$ see \cite[Theorem 4.1 (iii)]{CSS21}.  However, this strongly relies on the locality of the problem. In fact, the same construction does not work in the nonlocal setting $s\in(0,1)$.  To be more precise, let $J_{\R^n}$ be as in~\eqref{JOdef} and let
\begin{align*}
\mathcal{E}_{\rn}=\left\{u \in \mathcal{M}_{\rn}\::\: J_{\rn}^{\prime}(u) u^{+}=J_{\rn}^{\prime}(u)u^{-}=0 \text { and } u^{ \pm} \neq 0\right\},
\end{align*}
where $u^{+}(x)=\max \{u(x), 0\}, u^{-}(x)=\min \{u(x), 0\}$ and $\mathcal{M}_{\rn}$ is as in~\eqref{JOdef}.  Note that the set $\mathcal{E}_{\rn}$ contains all the $G$-invariant sign-changing solutions of the problem $(-\Delta)^s v = |v|^{2_s^*-2}v$ with $v\in D^{s}(\rn)$. This set is used in \cite{TWW15} to show the existence of least-energy nodal solutions for a fractional nonlinear problem in a bounded domain.  Now, let $\ell=2$ and $v:=v_{1,\infty}-v_{2,\infty}$, where $v_{1,\infty}$ and $v_{2,\infty}$ are the limit profiles given by Theorem~\ref{thm:phase_separation} for $s\in (0,1)$. Then,
\begin{align*}
J_{\rn}^{\prime}(v)v^{+} 
=\langle v^+,v^+\rangle+\langle v^-,v^+\rangle - \int_{\rn}\vert v^+\vert^{2^\ast_{s}}=\langle v^-,v^+\rangle>0,
\end{align*}
where we used that $\langle v^+,v^+\rangle=\langle v_{1,\infty},v_{1,\infty}\rangle=\int_{\rn}\vert v^+\vert^{2^\ast_{s}}$ and that
\begin{align*}
\langle v^-,v^+\rangle=-c_{N,s} \, p.v.\int_{\mathbb{R}^N}\int_{\mathbb{R}^N}\frac{v^+(x)v^-(y)}{\vert x - y\vert^{N+2s}}\, dx\, dy>0.
\end{align*}
As a consequence, $v=v_{1,\infty}-v_{2,\infty}$ cannot belong to $\mathcal{E}_{\rn}$ and therefore it cannot be a sign-changing solution. 
\end{remark}

\paragraph{Acknowledgments}
We thank Manuel Domínguez de la Iglesia for helpful discussions.

\appendix 

\section{On the conformal fractional Laplacian on the sphere} \label{ap:sec}

In this appendix we show Lemma \ref{lem:1}. To this end, we follow the calculations in \cite[Proposition 6.4.3]{DM18} with two differences. Firstly, we correct a small miscalculation in \cite[Proposition 6.4.3]{DM18} arising from a confusion in the use of the stereographic projection with respect to the north pole and the south pole. Secondly, in \cite[Proposition 6.4.3]{DM18} the computation of the constant $A_{N,s}$ relies on scattering theory. In our proof we obtain this constant by a direct computation.

\begin{proof}[Proof of Lemma \ref{lem:1}]
By \eqref{sigma}, we have for $x=\sigma(z)$ that $|x|^2 = \frac{|z'|^2}{(1+z_{N+1})^2} = \frac{1-z_{N+1}^2}{(1+z_{N+1})^2} = \frac{1-z_{N+1}}{1+z_{N+1}}.$ Hence,
\begin{align}\label{eq:4}
1+|x|^2 = \frac{2}{1+z_{N+1}}.
\end{align}
Letting $y = \sigma(\zeta)$ one also gets that
\begin{align*}
|x-y|^2
&= |x|^2+|y|^2-2x\cdot y
=\frac{1-z_{N+1}}{1+z_{N+1}}+\frac{1-\zeta_{N+1}}{1+\zeta_{N+1}}-\frac{2z'\cdot\zeta'}{(1+z_{N+1})(1+\zeta_{N+1})}\\
&= \frac{2-2z\cdot \zeta}{(1+z_{N+1})(1+\zeta_{N+1})}
= \frac{|z-\zeta|^2}{(1+z_{N+1})(1+\zeta_{N+1})}.
\end{align*}
From \eqref{eq:4}, $|z-\zeta|^2 = (1+z_{N+1})(1+\zeta_{N+1})|x-y|^2 = \frac{4|x-y|^2}{(1+|x|^2)(1+|y|^2)}.$ Let us now compute the left-hand side of \eqref{eq:2} using the previous identities and the change of variables given by $y = \sigma(\zeta)$ with determinant-Jacobian given by $\left(\frac{2}{1+|y|^2}\right)^N = \left(1+\zeta_{N+1}\right)^{-N}.$ Substituting this into the integral, we obtain
\begin{align*}
    &\psi_s^{-\frac{N+2s}{N-2s}}(-\Delta)^s(\psi_sv)(x)= 2^{-2s}c_{N,s}(1+|x|^2)^{N/2+s}p.v.\int_{\mathbb R^N} \frac{(1+|x|^2)^{-N/2+s}v(x)-(1+|y|^2)^{-N/2+s}v(y)}{|y-x|^{N+2s}}dy,\\
    &= c_{N,s}(1+z_{N+1})^{-N/2-s}p.v.\int_{\mathbb S^N} [(1+z_{N+1})^{N/2-s}u(z)-(1+\zeta_{N+1})^{N/2-s}u(\zeta)] \frac{(1+z_{N+1})^{N/2+s}(1+\zeta_{N+1})^{N/2+s-N}}{|\zeta-z|^{N+2s}}dV_g(\zeta)\\
    &= c_{N,s} \, p.v.\int_{\mathbb S^N} \left[\left(\frac{1+z_{N+1}}{1+\zeta_{N+1}}\right)^{N/2-s}u(z)-u(\zeta)\right]\frac{dV_g(\zeta)}{|\zeta-z|^{N+2s}}\\
    &= c_{N,s} \, p.v.\int_{\mathbb S^N} \frac{u(z)-u(\zeta)}{|\zeta-z|^{N+2s}}dV_g(\zeta) + u(z)c_{N,s} \, p.v.\int_{\mathbb S^N} \left[\left(\frac{1+z_{N+1}}{1+\zeta_{N+1}}\right)^{N/2-s}-1\right]\frac{dV_g(\zeta)}{|\zeta-z|^{N+2s}}.
\end{align*}

We focus now on the second term, returning the change of variables from before
\begin{align*}
    &c_{N,s} \, p.v.\int_{\mathbb S^N} \left[\left(\frac{1+z_{N+1}}{1+\zeta_{N+1}}\right)^{N/2-s}-1\right]\frac{dV_g(\zeta)}{|\zeta-z|^{N+2s}}\\
    &= 4^{-s} c_{N,s} \, p.v.\int_{\mathbb R^N} \left[\left(\frac{1+|y|^2}{1+|x|^2}\right)^{N/2-s}-1\right]\frac{(1+|x|^2)^{N/2+s}(1+|y|^2)^{N/2+s-N}}{|y-x|^{N+2s}}dy\\
    &= 4^{-s} c_{N,s}(1+|x|^2)^{N/2+s} p.v.\int_{\mathbb R^N} \left[(1+|x|^2)^{-N/2+s}-(1+|y|^2)^{-N/2+s}\right]\frac{dy}{|y-x|^{N+2s}}\\
    &= 4^{-s} (1+|x|^2)^{N/2+s} (-\Delta)^s (1+|x|^2)^{-N/2+s}. 
\end{align*}

To compute $(-\Delta)^s (1+|x|^2)^{-N/2+s}$, let $F_s:\rn\backslash \{0\}\to \r$ be the fundamental solution of $(-\Delta)^s$ in $\rn$, namely $F_s(x):=\kappa_{N,s}|x|^{2s-N}$ with $\kappa_{N,s}:=\frac{\Gamma(\frac{N}{2}-s)}{4^s\pi^\frac{N}{2}\Gamma(s)}$, see, for instance, \cite{B16} or \cite[Definition 5.6]{ajs}.  Let $P_s:\rn\times(0,\infty)\to \r$ denote the Poisson kernel for the operator $L_s := \Delta_{x,t} + \frac{1-2s}{t}\partial_t$, namely  $P_s(x,t) := p_{N,s}\frac{t^{2s}}{(t^2+|x|^2)^{\frac{N}{2}+s}}$ with $p_{N,s}:= \frac{\Gamma(\frac{N}{2}+s)}{\pi^\frac{N}{2}\Gamma(s)},$ see, for instance, \cite[eqn (2.3)]{MR2354493}. Note that, using polar coordinates,
\begin{align*}
\int_{\rn}\frac{t^{2s}}{(t^2+|x|^2)^{\frac{N}{2}+s}}\, dx
=\left(\frac{2\pi^\frac{N}{2}}{\Gamma(\frac{N}{2})}\right)\int_0^\infty \frac{\rho^{N-1}}{(1+\rho^2)^{\frac{N}{2}+s}}\, d\rho
=\frac{1}{p_{N,s}}\qquad \text{for all }t>0.
\end{align*}
A direct computation shows that the function $\Gamma_s(x,t) := \kappa_{N,s}(t^2+|x|^2)^{-N/2+s}$ satisfies that $L_s\Gamma_s=0$ in $\rn\times(0,\infty)$ and $\Gamma_s(x,0)=F_s(x)$ for $x\neq 0$. Then $\kappa_{N,s}(1+|x|^2)^{-N/2+s} =\Gamma_s(x,1)= [P_s(\cdot,1)\ast F_s](x)$, where $\ast$ denotes convolution. Hence,
\begin{align*}
\kappa_{N,s}(-\Delta)^s(1+|x|^2)^{-N/2+s}
&= (-\Delta)^s(P_s(\cdot,1)\ast F_s)
= P_s(\cdot,1)\ast (-\Delta)^sF_s
= P_s(\cdot,1)\ast\delta_0
= p_{N,s}(1+|x|^2)^{-N/2-s}.
\end{align*}
Note that $\frac{p_{N,s}}{\kappa_N,s} = 4^s\frac{\Gamma(N/2+s)}{\Gamma(N/2-s)}.$ Then, 
\begin{align*}
c_{N,s} \, p.v.\int_{\mathbb S^N} \left[\left(\frac{1+z_{N+1}}{1+\zeta_{N+1}}\right)^{N/2-s}-1\right]\frac{dV_g(\zeta)}{|\zeta-z|^{N+2s}}
&= 4^{-s} (1+|x|^2)^{N/2+s} (-\Delta)^s (1+|x|^2)^{-N/2+s}\\
&= \frac{\Gamma(N/2+s)}{\Gamma(N/2-s)}(1+|x|^2)^{N/2+s} (1+|x|^2)^{-N/2-s}=\frac{\Gamma(N/2+s)}{\Gamma(N/2-s)},
\end{align*}
as claimed.
\end{proof}


\begin{flushleft}

\textbf{Héctor A. Chang-Lara}\\
Matemáticas Básicas, Centro de Investigación en Matemáticas\\
Jalisco S/N, Col. Valenciana\\
C.P. 36000 Guanajuato, Mexico\\
\texttt{hector.chang@cimat.mx}

\medskip

\textbf{Juan Carlos Fernández}\\

Departamento de Matemáticas, Facultad de Ciencias\\
Universidad Nacional Autónoma de México\\
Circuito Exterior, Ciudad Universitaria\\
04510 Coyoacán, Ciudad de México, Mexico\\
\texttt{jcfmor@ciencias.unam.mx}

\medskip

\textbf{Alberto Saldaña}\\
Instituto de Matemáticas\\
Universidad Nacional Autónoma de México \\
Circuito Exterior, Ciudad Universitaria\\
04510 Coyoacán, Ciudad de México, Mexico\\
\texttt{alberto.saldana@im.unam.mx}
\medskip

\end{flushleft}

\end{document}